\documentclass[12pt]{amsart}
\usepackage{fullpage,enumerate,amsmath,amsthm,amsfonts,amssymb,amscd,amsopn,amsmath,mathrsfs,latexsym,graphicx,tikz-cd}
\usepackage[backend=bibtex]{biblatex}
\newtheorem{theorem}{Theorem}[section]
\newtheorem{mtheorem}[theorem]{Main Theorem}
\newtheorem{lemma}[theorem]{Lemma}
\newtheorem{proposition}[theorem]{Proposition}
\newtheorem{conjecture}[theorem]{Conjecture}
\theoremstyle{definition}
\newtheorem{definition}[theorem]{Definition}
\newtheorem{example}[theorem]{Example}
\theoremstyle{remark}
\newtheorem*{remark}{Remark}

\title{On the Classification of Schubert Varieties in Partial Flag Varieties}
\author{Yanjun Chen}
\address[Yanjun Chen]{Institut de Math\'ematique d'Orsay, Universit\'e Paris-Saclay, B\^atiment 307, rue Michel Magat, Orsay, France}
\email{yanjun.chen@universite-paris-saclay.fr}
\date{\today}

\begin{document}

\maketitle

\begin{abstract}
We generalize the classification of isomorphism classes of Schubert varieties in \cite{RS21} in complete flag varieties $G/B$ to a class of partial flag varieties $G/P$. In particular, we classify all Schubert varieties in $G/P$ where $P$ is a minimal parabolic subgroup and all Schubert surfaces. We also obtain several pairs of isomorphisms of Schubert varieties from folding the root system. This allows us to find an upper bound of the cardinality of isomorphism classes of Schubert three-folds.
\end{abstract}

\section{Introduction}
Schubert varieties form an extensively studied class of algebraic varieties whose properties are often characterized by combinatorics. The isomorphism problem for Schubert varieties, first raised by Develin, Martin, and Reiner in \cite{DMR07}, asks for a classification of all Schubert varieties up to algebraic isomorphism. In the same paper, they classified a class of smooth Schubert varieties in type $A$ partial flag varieties. Using Cartan equivalence, Richmond and Slofstra solved this problem for Schubert varieties in complete flag varieties in \cite{RS21}. On the other hand, Richmond, Țarigradschi, and Xu solved this problem for cominuscule Schubert varieties in \cite{RTX24} using labeled posets.

To describe our results, we set the following notations. We only consider Schubert varieties of finite types over $\mathbb{C}$ for simplicity. Let $G$ be a complex reductive Lie group, $T$ a maximal torus, and $B$ a Borel subgroup containing $T$. Then a root system is defined, with the correspondent set of simple reflections $S$, Cartan matrix $A = (a_{st})_{(s, t) \in S^2}$, and the Weyl group $W$. The pair $(W, S)$ forms a Coxeter system. A standard parabolic subgroup $P$ contains $B$ corresponding to a subset $I$ of $S$. We denote the flag variety $G/P$ by $X(A, I)$. The subset $I$ generates a subgroup $W_I \subset W$. In every coset $W/W_I$, we take the elements with minimal length and form a set $W^I$. Each $B$-orbit $BwP/P$ for some $w \in W^I$ is called a Schubert cell, and its Zariski closure is a Schubert variety, denoted by $X(w, A, I)$ since it is uniquely determined by the triple $(w, A, I)$. Let $\leq$ denote the Bruhat order for the Coxeter system $(W,S)$. We denote the \emph{support} of $w$ to be
\begin{equation*}
    S(w) = \{ s \in S|s \leq w \}.
\end{equation*}

\begin{definition}
    Let
    \begin{align*}
        A = (a_{st})_{(s, t) \in S^2} && \mbox{and} && A' = (a'_{s't'})_{(s', t') \in S'^2}
    \end{align*}
    be two Cartan matrices with associated Weyl groups $W$ and $W'$, and sets of simple reflections $S$ and $S'$, respectively. Let $I \subset S$ and $I' \subset S'$. For elements $w \in W^I$ and $w' \in W'^{I'}$, we say the triples $(w, A, I)$ and $(w', A', I')$ are \emph{Cartan equivalent} if there is a bijection $\tau: S(w) \to S(w')$ sending $S(w) \cap I$ to $S(w') \cap I'$, such that:
    \begin{enumerate}
        \item for some reduced word $w = s_1 \cdots s_k$, $w' = \tau(s_1) \cdots \tau(s_k)$ is also a reduced word;
        \item for any $t_1, t_2 \in S(w)$, $a_{t_1, t_2} = a'_{\tau(t_1), \tau(t_2)}$ whenever $t_1 t_2 \leq w$.
    \end{enumerate}
\end{definition}

When $I = I' = \varnothing$, our definition of Cartan equivalence coincides with the one given by Richmond and Slofstra. However, their Cartan equivalence does not behave well when $I$ or $I'$ is nonempty since they did not take account of the sets $I$ and $I'$; see \cite[Example 1.5]{RS21}. Using Cartan equivalence, they classified Schubert varieties of the form $X(w, A, \varnothing)$. They show that two such Schubert varieties $X(w, A, \varnothing), X(w', A’, \varnothing)$ are isomorphic if and only if $(w, A, I), (w', A', I')$ are Cartan equivalent. Motivated by their methods, we extend their work.

\begin{theorem}\label{thm-iso->car}
    Assume that $|S(w) \cap I| \leq 1$ and $|S(w') \cap I'| \leq 1$. If $X(w, A, \{ s \})$ and $X(w', A', \{ s' \})$ are algebraically isomorphic, then $(w, A, I)$ and $(w', A', I')$ are Cartan equivalent.
\end{theorem}

\begin{theorem}\label{thm-car->iso}
    If $(w, A, I)$ and $(w', A', I')$ are Cartan equivalent, then $X(w, A, \{ s \})$ and $X(w', A', \{ s' \})$ are algebraically isomorphic.
\end{theorem}

The proof of the above theorems will be given in Sections \ref{co} and \ref{con iso}, respectively.

Since we allow the subset $I$ to satisfy the condition $|S(w) \cap I| \leq 1$, a larger class of Schubert varieties is classified. It also completely solves the isomorphism problem of Schubert varieties in flag varieties corresponding to minimal parabolic subgroups, in which case, $|I| = 1$.

Schubert varieties of dimension zero and one are isomorphic to a point and the projective line $\mathbb{P}^1$, respectively. Applying these two theorems, we can easily classify the Schubert surfaces without explicit computation.

\begin{theorem}[Classification of Schubert Surfaces]\label{cla}
    There are seven isomorphism classes of Schubert surfaces. Precisely, any Schubert surface is isomorphic to one and exactly one of the following surfaces:
    \begin{enumerate}
        \item product of projective lines $\mathbb{P}^1 \times \mathbb{P}^1$,
        \item projective plane $\mathbb{P}^2$,
        \item the $n$-th Hirzebruch surface $\Sigma_n$ for $n = 1, 2, 3$,
        \item the cone over a smooth conic, and
        \item a Schubert surface of type $G_2$.
    \end{enumerate}
\end{theorem}

The proof will be given at the end of Section \ref{EC}. In the proof, we will construct of the Schubert surface of type $G_2$ stated above.

We say two Schubert varieties $X(w, A, I)$ and $X(w', A', I')$ are \emph{equally supported} if $|S(w)| = |S(w')|$. By definition, if $(w, A, I)$ and $(w', A', I')$ are Cartan equivalent, then $X(w, A, I)$ and $X(w', A', I')$ are equally supported.

\begin{conjecture}\label{conj}
    Assume that $X(w, A, I)$ and $X(w', A', I')$ are equally supported. The following are equivalent:
    \begin{enumerate}
        \item the Schubert varieties $X(w, A, I)$ and $X(w', A', I')$ are algebraically isomorphic;
        \item $(w, A, I)$ and $(w', A', I')$ are Cartan equivalent.
    \end{enumerate}
\end{conjecture}

By Theorem \ref{thm-car->iso}, it remains to show that the equally supported, isomorphic Schubert pairs are Cartan equivalent. We will leave it for future study.

When the dimensions of Schubert varieties are greater than two, the isomorphic pairs that are not equally supported arise. An example is given in \cite[Example 1.5]{RS21}. Hence the notion of Cartan equivalence fails in this case. The example comes from a more general setting, called folding the root systems. We construct several classes of isomorphic pairs from in this setting.

\begin{theorem}\label{fold-theorem}
Let $R$ be a simply laced root system with simple roots $\Delta$, and let $\tau$ be a diagram automorphism of $(R, \Delta)$. Let
\begin{align*}
A = (a_{s_i s_j})_{(s_i, s_j) \in S^2} && \mbox{and} && A^\tau = (a_{t_k t_l})_{(t_k, t_l) \in {S^\tau}^2}
\end{align*}
be the Cartan matrices of the root systems $(R, \Delta)$ and $(R^\tau, \Delta^\tau)$, respectively. Let $I \subset S$ and $I^\tau \subset S^\tau$ and take two elements $w \in W^I$ and $w^\tau \in {W^\tau}^{I^\tau}$. Assume that the map
\begin{align*}
[1, w^\tau]^{I^\tau} \to [1, w]^I, && v \mapsto v^{\min}
\end{align*}
is bijective, where $v^{\min}$ is the minimal representative of $v$ in $W^I$. Then $X(w, A, I)$ and $X(w^\tau, A^\tau, I^\tau)$ are isomorphic as algebraic varieties.
\end{theorem}

We will recall the notion of folding and define the necessary notations in Section \ref{folding-recall}. The proof of this theorem will be given in Section \ref{con iso 2}.

However, no converse theorem of Theorem \ref{fold-theorem} like Theorem \ref{thm-iso->car} is known, so the classification of higher dimensional Schubert varieties is impossible if we restrict ourselves to combinatorial method yet. Nevertheless, it is still possible to give an upper bound of the cardinality of isomorphism classes of Schubert varieties with fixed dimension. We give the bound in dimension three to illustrate our idea.

\begin{theorem}\label{thm-schu 3fold}
There are at most $34$ isomorphism classes of Schubert three-folds.
\end{theorem}

The proof will be given at the end of Section \ref{explicit type}. The reader may find the explicit cardinality by explicit computation, but we refuse to do this because we want to keep the combinatorial style.

\subsection{Outline of the Paper}
In Section \ref{EC}, we will show some illustrative examples of Theorems \ref{thm-iso->car} and \ref{thm-car->iso}. In particular, we will construct all seven Schubert surfaces as claim in Theorem \ref{cla}. At the end of this section, we will prove Theorem \ref{cla}.

The following two sections will give type-independent proofs of Theorems \ref{thm-iso->car} and \ref{thm-car->iso}, respectively. In Section \ref{co}, we will study the structure of the (singular) cohomology ring of a Schubert variety to prove Theorem \ref{thm-iso->car}. In fact, we will show that, if the cohomology rings of the Schubert varieties $X(w, A, I)$ and $X(w', A', I')$ (satisfying the condition in Theorem \ref{thm-iso->car}) are isomorphic as graded rings, then $(w, A, I)$ and $(w', A', I')$ are Cartan equivalent. The cohomology ring has a canonical basis, namely the \emph{Schubert basis}, whose product formula is given by Chevalley’s formula \cite[Lemma 8.1]{FW01}. We will see that many constructions in the theory of Coxeter groups can be rephrased in terms of the Schubert basis. Our computation and discussion will follow the methods of Richmond and Slofstra in \cite[Section 4]{RS21}, but will be deeper and more explicit than theirs.

In Section \ref{con iso}, we will prove Theorem \ref{thm-car->iso}. We will explicit construct the isomorphism, first in the Lie algebra level (see Lemmas \ref{Lie homomor} and \ref{varphi}), and then using exponential map to lift the Lie algebra homomorphisms into an algebraic morphism of varieties (see Lemma \ref{pi}). This morphism is bijective, so it is an isomorphism since Schubert varieties are normal. Our proof generalizes the one given in \cite[Section 3]{RS21}.

The remaining part of this paper focuses on the non-equally supported pairs of Schubert varieties. In Section \ref{folding-recall}, we recall the notion of folding the root system by an automorphism. The literature does not fully state these materials, so we present them for complement. We will use \cite{St08} as our basic reference, but we will change their notations to present the paper in a more coherent form.

In Section \ref{explicit type}, we will explain Theorem \ref{fold-theorem} by examples in different types of folding. Then we will turn to prove Theorem \ref{thm-schu 3fold}.

Finally, in Section \ref{con iso 2}, we will prove Theorem \ref{fold-theorem}. We will imitate the method in \cite[Section 3]{RS21}. However, the proof will be much more difficult in this case.

\subsection*{Acknowledgment}
The author thanks Weihong Xu for introducing him to the isomorphism problem, providing him with relevant materials, and guiding him in writing the thesis. He also thanks Jing-Song Huang and Rui Xiong for their instruction on representation theory. He was partially supported by NSFC grant 12426507.

\section{Classification of Schubert Surfaces}\label{EC}
We provide several examples of Schubert varieties to illustrate our context. We will use $s_i$ to denote the simple reflection corresponding to the $i$-th row or column of the Cartan matrix.

\begin{example}
Considering the three Schubert varieties
\begin{align*}
X \left( s_1 s_2 s_3 s_4, \begin{pmatrix} 2 & -1 & 0 & 0 \\ -1 & 2 & -1 & 0 \\ 0 & -1 & 2 & -1 \\ 0 & 0 & -1 & 2 \end{pmatrix}, \{ s_1 \} \right), \\
X \left( s_1 s_2 s_3 s_4, \begin{pmatrix} 2 & -1 & 0 & 0 \\ -1 & 2 & -1 & 0 \\ 0 & -1 & 2 & -1 \\ 0 & 0 & -2 & 2 \end{pmatrix}, \{ s_1 \} \right), \\
X \left( s_1 s_2 s_3 s_4, \begin{pmatrix} 2 & -1 & 0 & 0 \\ -1 & 2 & -1 & 0 \\ 0 & -2 & 2 & -1 \\ 0 & 0 & -1 & 2 \end{pmatrix}, \{ s_1 \} \right),
\end{align*}
of types $A_4$, $B_4$, and $F_4$, respectively. The upper triangular parts of three Cartan matrices are the same, so by the main theorem, they are isomorphic.
\end{example}

We call a two-dimensional Schubert variety a \emph{Schubert surface}. For $n = 1, 2, 3$, we let $M_n$ denote the Cartan matrix
\begin{equation*}
M_n = \begin{pmatrix} 2 & -1 \\ -n & 2 \end{pmatrix}.
\end{equation*}
In other words, $M_1$, $M_2$, and $M_3$ are the Cartan matrices of type $A_2$, $B_2$, and $G_2$, respectively. Note that $s_1$ and $s_2$ correspond to the long and the short simple root, respectively.

\begin{example}\label{12}
Consider the Schubert surfaces $X(s_1 s_2, M_n, \{ s_1 \})$. Since the entries $a_{12} = -1$ of Cartan matrices are equal, these three surfaces are isomorphic by Theorem \ref{thm-car->iso}. Taking $n = 1$, $X(s_1 s_2, M_1, \{ s_1 \})$ is a Schubert surface in the two dimensional flag variety $Gr(1, 3) = \mathbb{P}^2$, so $X(s_1 s_2, M_n, \{ s_1 \})$ is isomorphic to $\mathbb{P}^2$. Similarly, the Schubert surfaces $X(s_1 s_2, M_n, \varnothing)$ are isomorphic for each $n$. By \cite[Example 1.1]{RS21}, they are isomorphic to the \emph{first Hirzebruch surface} $\Sigma_1$.
\end{example}

\begin{example}\label{21,1}
Consider the Schubert surfaces $X(s_2 s_1, M_n, \{ s_2 \})$. The corresponding entries $a_{21} = -n$, so these three varieties are \emph{not} isomorphic to each other by Theorem \ref{thm-iso->car}. Note that geometrically, $X(s_2 s_1, M_1, \{ s_2 \})$ is the same as $X(s_1 s_2, M_1, \{ s_1 \})$ up to symmetry, so it is also isomorphic to $\mathbb{P}^2$. On the other hand, it can be shown that $X(s_2 s_1, M_2, \{ s_2 \})$ is the cone over a smooth conic. In fact, the flag variety $X(s_2 s_1, M_2)$ is a smooth quadric hypersurface in $\mathbb{P}^4$ defined by the equation $x_1 x_5 + x_2 x_4 + x_3^2 = 0$, and the Schubert variety $X(s_2 s_1, M_2, \{ s_2 \})$ lies in the hyperplane $x_5 = 0$. Hence this Schubert surface is defined by the equation $x_2 x_4 + x_3^2 = 0$ in $\mathbb{P}^3$, which is the cone over a smooth conic. Finally, the Schubert surface $X(s_2 s_1, M_3, \{ s_2 \})$ is of type $G_2$, the last surface stated in Theorem \ref{cla}.
\end{example}

\begin{example}\label{21,0}
Consider the Schubert surfaces $X(s_2 s_1, M_n, \varnothing)$. The corresponding entries $a_{21} = -n$, so these four surfaces are \emph{not} isomorphic to each other. As mentioned in \cite[Example 1.1]{RS21}, $X(s_2 s_1, M_n, \varnothing)$ is isomorphic to the \emph{$n$-th Hirzebruch surface} $\Sigma_n$.
\end{example}

\begin{example}\label{*}
Consider the Schubert surface $X(s_1 s_2, 2I_2, \varnothing)$, where $I_2$ is the $2 \times 2$ identity matrix. The corresponding entry $a_{12} = 0$, so it is different from any Schubert surfaces mentioned above by Theorem \ref{thm-iso->car}. The Cartan matrix is of type $A_1 \times A_1$, so the flag variety can be identified with $\mathbb{P}^1 \times \mathbb{P}^1$. Comparing the dimensions, we conclude that $X(s_1 s_2, 2I_2, \varnothing)$ is $\mathbb{P}^1 \times \mathbb{P}^1$.
\end{example}

The Schubert surfaces in the above four examples are all possible Schubert surfaces. Now we prove this classification theorem (Theorem \ref{cla}).

\begin{proof}[Proof of Theorem \ref{cla}]
    Recall that the dimension of a Schubert variety $X(w, A, I)$ is the length $\ell(w)$ of $w \in W^I$ so that we can write $w = st \in W^I$ as the product of two simple reflections $s$ and $t$. \cite[Lemma 4.8]{RS16} implies that, without loss of generality, we can assume that $S(w) = S$. In other words, $S = \{ s, t \}$ and $A$ has rank two. Since a root system of rank two has type $A_1 \times A_1$, $A_2$, $B_2$, or $G_2$, after appropriately indexing rows and columns, we have $A = 2I_2$ or $A = M_n$ for $n = 1, 2, 3$. Hence $I$ is empty or a singleton. It follows that there are only finitely many possibilities for $X(w, A, I)$:
    \begin{enumerate}
        \item $X(s_1 s_2, 2I_2, \varnothing)$, isomorphic to $\mathbb{P}^1 \times \mathbb{P}^1$;
        \item $X(s_1 s_2, M_n, \varnothing)(n = 1, 2, 3)$ and $X(s_2 s_1, M_1, \varnothing)$, isomorphic to $\Sigma_1$;
        \item $X(s_2 s_1, M_n, \varnothing)(n = 2, 3)$, isomorphic to $\Sigma_n$;
        \item $X(s_1 s_2, M_n, \{ s_1 \})(n = 1, 2, 3)$ and $X(s_2 s_1, M_1, \{ s_2 \})$, isomorphic to $\mathbb{P}^2$;
        \item $X(s_2 s_1, M_2, \{ s_2 \})$, isomorphic to the cone over a smooth conic;
        \item $X(s_2 s_1, M_3, \{ s_2 \})$.
    \end{enumerate}
These varieties are not Cartan equivalent. Hence there are 7 isomorphism classes, as required.
\end{proof}

\section{The Cohomology Rings of Schubert Varieties}\label{co}
In this section, we will use the cohomology ring to study the intrinsic properties of a Schubert variety. We will see that most combinatorial information of a Schubert variety can be obtained from its cohomology ring. Some results have been proven in \cite[Section 4]{RS21} for the case $I$ being an empty set, and the proofs are generalized. We will fix a Schubert variety $X(w, A, I)$ with $w \in W^I$. Due to \cite[Lemma 4.8]{RS16}, we assume that $S = S(w)$.

\subsection{General Situation}\label{ge ca}
\begin{definition}\label{sch basis}
There is a $\mathbb{Z}$-basis $\sigma_v$, indexed by $v \in [1, w]^I$, for the integral cohomology group $H^*(X(w, A, I))$. We call this basis to be \emph{Schubert basis}, denoted by $\Sigma(w, A, I)$. The collection of Schubert basis elements of degree two is denoted by $\Sigma_1(w, A, I)$.
\end{definition}

The cup product of $H^*(X(w, A, I))$ is given by the \emph{Chevalley's formula} in \cite[Lemma 8.1]{FW01}. We let $v^{\min} \in W^I$ be the minimal length representative in the coset $vW_I$ for $v \in W$. The map $w \mapsto w^{\min}$ is order-preserving, as stated in \cite[Proposition 2.5.1]{BB05}. The Cartan matrix defines a root system $R$, a choice of positive/negative roots $R = R^+ \cup R^-$, and the set $\Delta$ of simple roots. The set $I$ defines a subset of simple roots $\Delta_I$. We write the root subsystem spanned by $\Delta_I$ as $R_I$.

\begin{lemma}[Chevalley's formula]
Suppose that $\alpha \in \Delta \setminus \Delta_I$ and $w \in W^I$. Then
\begin{equation*}
\sigma_{s_\alpha} \sigma_w = \sum \omega_\alpha(\beta^\vee) \sigma_{(w s_\beta)^{\min}},
\end{equation*}
the sum over all positive roots $\beta$ such that $\ell((w s_\beta)^{\min}) = \ell(w) + 1$. Here, $\omega_\alpha$ is the fundamental weight corresponding to $\alpha$, and $\beta^\vee = 2 \beta / \langle \beta, \beta \rangle$ is a coroot.
\end{lemma}

\begin{remark}
For any positive root $\beta$, we have either $w s_\beta \leq w$ or $w s_\beta > w$. If the former one holds, then we have $(w s_\beta)^{\min} \leq w^{\min} = w$, which is impossible. Hence $w s_\beta > w$.
\end{remark}

The following lemma is essentially \cite[Proposition 4.1]{RS21}, where the case $I = \varnothing$ was proven. Nevertheless, the statement holds in a more general case, and the proof can be extended to our context when replacing $X(w, A)$ and $X(w', A')$ in their paper by $X(w, A, I)$ and $X(w', A', I')$, respectively.

\begin{lemma}\label{basis}
Suppose that $f: X(w, A, I) \to X(w', A', I')$ is an algebraic isomorphism of Schubert varieties. Then the induced map
\begin{equation*}
\varphi^*: H^*(X(w, A, I)) \to H^*(X(w', A', I'))
\end{equation*}
is a graded ring isomorphism that identifies Schubert bases $\Sigma(w, A, I) \to \Sigma(w', A', I')$ and $\Sigma_1(w, A, I) \to \Sigma_1(w', A', I')$.
\end{lemma}

In other words, both $\Sigma(w, A, I)$ and $\Sigma_1(w, A, I)$ are determined by the cohomology group. From this definition, we have a bijection
\begin{align*}
i: [1, w]^I \to \Sigma(w, A, I), && u \mapsto \sigma_u.
\end{align*}
Each basis element $\sigma_u$ has degree $\deg \sigma_u = 2 \ell(u)$, i.e.,
\begin{equation*}
\sigma_u \in H^{2 \ell(u)}(X(w, A, I)).
\end{equation*}
Recall that the set of simple reflections in $[1, w]^I$ is just $S(w) \setminus I$. Hence the restriction of $i$ gives a one-one correspondence
\begin{align*}
i: S(w) \setminus I \to \Sigma_1(w, A, I), && s \mapsto \sigma_s.
\end{align*}

\begin{definition}
    For any
    \begin{equation*}
        \sigma = \sum_{v \in W^I} a_v \sigma_v \in H^*(X(w, A, I)),
    \end{equation*}
    we define its \emph{support}, denoted by $\Sigma(\sigma)$, to be the collection of $\sigma_v$ with $a_v \neq 0$. Let $\prec$ be the partial order on $\Sigma(w, A, I)$, the Schubert basis, generated by the relation $\sigma_u \prec \sigma_v$ if $\sigma_v \in \Sigma(\sigma_s \sigma_u)$ for some $\sigma_s \in \Sigma_1(w, A, I)$.
\end{definition}

\begin{lemma}\label{^I}
The bijection $i: [1, w]^I \to \Sigma(w, A, I)$ is a poset isomorphism.
\end{lemma}

\begin{proof}
It remains to prove that, for $u, v \in [1, w]^I$, $u < v$ if and only if $\sigma_u \prec \sigma_v$, and it suffices to consider the covering relations. If $u < v$ and $\ell(v) = \ell(u) + 1$, then $v = u s_\beta$ for some positive root $\beta$. Since $u, v \in W^I$, $\beta \notin R_I$. Hence there exists a simple root $\alpha \in \Delta \setminus \Delta_I$ such that $\omega_\alpha(\beta^\vee) \neq 0$. It follows that $\sigma_v \in \Sigma(\sigma_{s_\alpha} \sigma_u)$.

Conversely, if $\sigma_v \in \Sigma(\sigma_{s} \sigma_u)$ for some $\sigma_s \in \Sigma_1(w, A, I)$, then $v = (u s_\beta)^{\min}$ for some positive root $\beta$. By the remark below \cite[Lemma 4.8]{RS16}, we have $u s_\beta > u$, and it follows that $u < v$.
\end{proof}

\begin{definition}
    Given a subset $J'$ of $\Sigma_1(w, A, I)$, let $H^{J'}$ be the subring of $H^*(X(w, A, I))$ generated by
    \begin{equation*}
        \sigma_s \in \Sigma_1(w, A, I) \setminus J',
    \end{equation*}
    and let
    \begin{equation*}
        \Sigma(w, A, I)^{J'} = \bigcup_{\sigma \in H^{J'}} \Sigma(\sigma).
    \end{equation*}
\end{definition}

We define the \emph{inversion set} of $v \in W$ to be $I(v) = R^+ \cap v(R^-)$.

\begin{lemma}\label{^J}
Suppose that there is a subset $J \subset S \setminus I$. Then the poset isomorphism $i: [1, w]^I \to \Sigma(w, A, I)$ restricts to another poset isomorphism
\begin{equation*}
i^J: [1, w]^{I \cup J} \to \Sigma(w, A, I)^{i(J)}.
\end{equation*}
\end{lemma}

\begin{proof}
It suffices to prove that $i^J$ is a bijection. First, we show that $\mbox{im } i^J \subset \Sigma(w, A, I)^{i(J)}$ by induction on length. Suppose that $v \in [1, w]^{I \cup J}$ and $v = s_1 \cdots s_k$ is a reduced word. Let $u = s_1 v = s_2 \cdots s_k \in [1, w]^{I \cup J}$. Then since $\ell(u) < \ell(v)$, by induction, $\sigma_u \in \Sigma(w, A, I)^{i(J)}$. By \cite[Lemma 1.6]{Hum92}, we conclude that $u^{-1}(\alpha_{s_1}) \in R^+$. Then
\begin{equation*}
u^{-1}(\alpha_{s_1}) = v^{-1} s_1(\alpha_{s_1}) = v^{-1}(- \alpha_{s_1}) \in I(v^{-1}).
\end{equation*}
Since $v \in W^{I \cup J}$, by \cite[Lemma 1.6]{Hum92} again we have $I(v^{-1}) \subset R^+ \setminus R_{I \cup J}^+$. Thus, there exists a simple root $\alpha \in \Delta \setminus \Delta_{I \cup J}$, such that $\omega_\alpha(u^{-1}(\alpha_{s_1}^\vee)) \neq 0$. Then
\begin{equation*}
\sigma_v \in \Sigma(\sigma_{s_\alpha} \sigma_u) \subset \Sigma(w, A, I)^{i(J)}
\end{equation*}
since $u^{-1} v = u^{-1} s_1 u$ is the reflection corresponding to the positive root $u^{-1}(\alpha_{s_1})$. Hence the inclusion $\mbox{im } i^J \subset \Sigma(w, A, I)^{i(J)}$ holds.

Conversely, suppose that $u \in [1, w]^{I \cup J}$, $\alpha \in \Delta \setminus \Delta_{I \cup J}$ and $\sigma_v \in \Sigma(\sigma_{s_\alpha} \sigma_u)$. We need to show that $v \in W^{I \cup J}$. Then the lemma follows from the similar induction process. If $v \notin W^{I \cup J}$, then there exists a reduced word $v = s_1 \cdots s_k$ with $s_k \in {I \cup J}$. Since $u < v$ and $u \in W^{I \cup J}$, $u$ must have the form $u = s_1 \cdots s_{k-1}$. But $\omega_\alpha(\alpha_{s_k}^\vee) = 0$, which contradicts to the assumption $\sigma_v \in \Sigma(\sigma_{s_\alpha} \sigma_u)$.
\end{proof}

Now we consider the right descent set. We define the \emph{right descent set} of $v \in W$ to be 
\begin{equation*}
D_R(v) = \{ s \in S|\ell(vs) = \ell(v) - 1 \} = \{ s \in S|v \notin W^{\{ s \}} \}.
\end{equation*}
Similarly, we define the \emph{right descent set} of $\sigma_v$ to be
\begin{equation*}
D_R(\sigma_v) = \left\{ \sigma_s \in \Sigma_1(w, A, I)|\sigma_v \notin \Sigma(w, A, I)^{\{ \sigma_s \}} \right\}.
\end{equation*}

\begin{lemma}
Suppose that $v \in [1, w]^I$. Then the poset isomorphism $i: [1, w]^I \to \Sigma(w, A, I)$ restricts to a bijection $i_v: D_R(v) \to D_R(\sigma_v)$.
\end{lemma}

\begin{proof}
Since $v \in W^I$, $s \in D_R(v)$ if and only if $v \notin W^{I \cup \{ s \}}$. Now the result follows from Lemma \ref{^J}.
\end{proof}

\begin{lemma}\label{not exist}
Suppose that $\sigma_v \in \Sigma(w, A, I)$ and $\sigma_s \in D_R(\sigma_v)$. Then there exists a unique maximal element $\sigma_u \prec \sigma_v$ such that $\sigma_u \in \Sigma(w, A, I)^{\{ \sigma_s \}}$.
\end{lemma}

\begin{proof}
Under the identification $i$, $u$ is the maximal element satisfies $u < v$ and $u \in W^{I \cup \{ s \}}$. In other words, $u$ is the minimal representative of $v$ in the coset of $W/W_{I \cup \{ s \}}$. Hence $u$ exists and is unique, and so does and is $\sigma_u$.
\end{proof}

We obtain the following proposition as a corollary, which gives a necessary condition.

\begin{proposition}\label{not}
If $S(w) \cap I$ is empty but $S(w') \cap I'$ is not, then $X(w, A, I)$ and $X(w', A', I')$ are not isomorphic as algebraic varieties.
\end{proposition}

\begin{proof}
Using \cite[Lemma 4.8]{RS16}, we can assume that $w, w'$ are fully supported. Hence $I = \varnothing \neq I'$. We can find two simple reflections $s', t' \in S'$ such that $s't' \in (1, w']^{I'}$ and $s' \in I$. In fact, taking a reduced word $w' = s'_1 \cdots s'_k$, we let $s' = s'_i$, where $i$ is largest integer such that $s'_i \in I'$, and $t' = s'_j$ with $j > i$ such that $s'_i$ and $s'_j$ cannot commute. Then $s't' = s'_i s'_j \in (1, w']^{I'}$. The integer $j$ exists since otherwise, $s_i$ commutes with all $s'_j$ for all $j > i$, which implies
\begin{equation*}
w' = s'_1 \cdots \widehat{s'_i} \cdots s'_k s'_i \notin W'^{I'},
\end{equation*}
contradicting to the assumption that $w' \in W'^{I'}$.

Suppose that $X(w, A, \varnothing)$ and $X(w', A', I')$ are algebraically isomorphic. Then the set $\Sigma_1(w, A, I)$ (resp. the poset $\Sigma(w, A, I)$) can be identified with the set $\Sigma_1(w', A', I')$ (resp. the poset $\Sigma(w', A', I')$). Hence there is a pair
\begin{equation*}
(\sigma_t, \sigma_{st}) \in \Sigma_1(w, A, I) \times \Sigma(w, A, I)
\end{equation*}
of Schubert basis elements corresponding to the pair
\begin{equation*}
(\sigma_{t'}, \sigma_{s' t'}) \in \Sigma_1(w', A', I') \times \Sigma(w', A', I').
\end{equation*}
Consider the unique maximal element $\sigma_u \prec \sigma_{st}$ (resp. $\sigma_{u'} \prec \sigma_{s' t'}$) such that $\sigma_u \in \Sigma(w, A, I)^{\{ \sigma_t \}}$ (resp. $\sigma_{u'} \in \Sigma(w', A', I')^{\{ \sigma_{t'} \}}$), which exists by Lemma \ref{not exist}. But $u = s$ and $u' = 1$, so
\begin{equation*}
\deg \sigma_u = 2 \neq 0 = \deg \sigma_{u'},
\end{equation*}
which is a contradiction.
\end{proof}

For the purpose of this section, we compute some entries of the Cartan matrix from the cohomology ring.

\begin{lemma}\label{C1}
Let $s \in S$ and $t \in S \setminus I$ with $s \neq t$. If $st \in [t, w]^I$, then the coefficient of $\sigma_{st}$ in $\sigma_t^2$ is $- a_{st}$.
\end{lemma}

\begin{proof}
By Chevalley's formula, the coefficient of $\sigma_{st}$ in $\sigma_t^2$ is 
\begin{equation*}
\omega_{\alpha_t}(t(\alpha_s^\vee)) = \omega_{\alpha_t}(\alpha_s^\vee - a_{st} \alpha_t^\vee) = - a_{st}.
\end{equation*}
\end{proof}

\begin{lemma}\label{C2}
Let $r, s \in S$ and $t \in S \setminus I$ with $rst \in [1, w]^I$ being a reduced word. Then the coefficient of $\sigma_{rst}$ in $\sigma_t \sigma_{st}$ is $\delta_{rt} - a_{rt} + a_{rs} a_{st}$, where $\delta_{rt} = 1$ if $r = t$, and $\delta_{rt} = 0$ if $r \neq t$.
\end{lemma}

\begin{proof}
By Chevalley's formula, the coefficient of $\sigma_{rst}$ in $\sigma_t \sigma_{st}$ is
\begin{align*}
\omega_{\alpha_t}(ts(\alpha_r^\vee)) & = \omega_{\alpha_t}(t(\alpha_r^\vee - a_{rs} \alpha_s^\vee)) = \omega_{\alpha_t}(\alpha_r^\vee - a_{rs} \alpha_s^\vee - (a_{rt} - a_{rs} a_{st}) \alpha_t^\vee) \\
& = \delta_{rt} - (a_{rt} - a_{rs} a_{st}) = \delta_{rt} - a_{rt} + a_{rs} a_{st}.
\end{align*}
\end{proof}

\subsection{Case $I = \{ s \}$}\label{sp ca}
We further assume that $I = \{ s \}$. We are primarily interested in this situation.

We now define the reduced word of $\sigma_v \in \Sigma(w, A, I)$ corresponding to the reduced word of $v$. Let $\sigma_s$ be a symbol and define
\begin{equation*}
\widetilde{\Sigma}_1(w, A, I) = \Sigma_1(w, A, I) \cup \{ \sigma_s \},
\end{equation*}
and extend the bijection $S \setminus I \to \Sigma_1(w, A, I)$ to a bijection $S \to \widetilde{\Sigma}_1(w, A, I)$ by sending $s$ to the symbol $\sigma_s$. We define the \emph{reduced words} of $\sigma_v \in \Sigma(w, A, I)$ inductively. First, define the reduced word of $\sigma_1 \in H^0(X(w, A, \{ s \}))$ to be the singleton of the empty word. Next, suppose that $\sigma_v \in \Sigma(w, A, I)$ with $\sigma_t \in D_R(\sigma_v)$. Take $\sigma_u$ to be the unique maximal element described in Lemma \ref{not exist}, and assume that $\deg \sigma_v - \deg \sigma_u = 2n$ for $n \in \mathbb{N}$. Then we call the sequence
\begin{equation*}
(\sigma_{s_1}, \cdots, \sigma_{s_k}, \cdots, \sigma_s, \sigma_t, \sigma_s, \sigma_t)
\end{equation*}
a reduced word of $\sigma_v$, where $(\sigma_{s_1}, \cdots, \sigma_{s_k})$ is a reduced word of $\sigma_u$, and the remaining part $(\cdots, \sigma_s, \sigma_t, \sigma_s, \sigma_t)$ of the sequence is an alternating sequence of $\sigma_t$ and $\sigma_s$ of length $n$. We also denote $RW(\sigma_v)$ as the set of reduced words of $\sigma_v$, which is nonempty. We write $RW(u)$ as the set of reduced words of $u \in W$. The following lemma is clear.

\begin{lemma}\label{red}
Suppose that $\sigma_v \in \Sigma(w, A, I)$. Then the bijection $S \to \widetilde{\Sigma}_1$ induces an inclusion
\begin{align*}
RW(\sigma_u) \hookrightarrow RW(u), && (\sigma_{s_1}, \cdots, \sigma_{s_k}) \mapsto (s_1, \cdots, s_k).
\end{align*}
\end{lemma}

The above lemmas allow us to prove the following proposition.

\begin{proposition}\label{neces}
Let $A = (a_{st})_{(s, t) \in S^2}$ and $A' = (a'_{s't'})_{(s', t') \in S'^2}$ be two Cartan matrices with associated Weyl groups $W$ and $W'$, and sets of simple reflections $S$ and $S'$, respectively. Let $s \in S$ and $s' \in S'$. Take two fully supported elements $w \in W^{\{ s \}}$ and $w' \in W'^{\{ s' \}}$. If $X(w, A, \{ s \})$ and $X(w', A', \{ s' \})$ are algebraically isomorphic, then $(w, A, \{ s \})$ and $(w', A', \{ s' \})$ are Cartan equivalent.
\end{proposition}

\begin{proof}
To simplify notations, we write $I = \{ s \}$ and $I' = \{ s' \}$. We denote $i'$ as the identification $[1, w']^{I'} \to \Sigma(w', A', I')$. We define $\tau$ to be the composition of bijections
\begin{equation*}
S \stackrel{i}{\longrightarrow} \widetilde{\Sigma}_1(w, A, I) \longrightarrow \widetilde{\Sigma}_1(w', A', I') \stackrel{i'^{-1}}{\longrightarrow} S'.
\end{equation*}
Then $\tau$ identifies the reduced words by Lemma \ref{red}. Hence the first condition of Cartan equivalence is proven.

Now we check that the second condition holds valid.

Suppose that $t_1, t_2 \in S$ such that $t_1 t_2 \leq w$. If $t_1 = t_2$, then $a_{t_1 t_2} = 2 = a'_{\tau(t_1) \tau(t_2)}$. If $t_1 t_2 \in (1, w]^I$, then Lemma \ref{C1} implies that the coefficient of $\sigma_{t_1 t_2}$ and $\sigma_{\tau(t_1) \tau(t_2)}$ in $\sigma_{t_2}^2$ and $\sigma_{\tau(t_2)}^2$ are $- a_{t_1 t_2}$ and $- a'_{\tau(t_1) \tau(t_2)}$, respectively. Hence $a_{t_1 t_2} = a'_{\tau(t_1) \tau(t_2)}$.

Suppose that $t_1 t_2 \notin [1, w]^I$ for the remaining case. Then $t_1 t_2 = rs$ for some $r \in S \setminus I$. Taking a reduced word $w = s_1 \cdots s_k$ corresponding to a reduced word of $\sigma_w$, let $i$ be the maximal integer such that $s_i = s$. We further assume that the reduced word of $w$ is taken to satisfy that $i$ is maximal among all such reduced words. Then $t = s_{i+1} \in S \setminus I$ satisfies $st \neq ts$. If $rst$ is not reduced, then $rst = s$ since $rst \neq r$ or $t$. It follows that $r = t$ and hence $m_{st} = 2$, which contradicts the assumption that $s$ and $t$ cannot commute. Thus, $rst$ is reduced. If $rst \notin W^I$, then the nonreduced word $rsts = rt$ or $st$. The former is impossible since otherwise, we will again obtain that $m_{st} = 2$, which is a contradiction. Hence $rsts = st$. It follows that the simple reflection $r \in W_{\{ s, t \}}$, so $r = t$, and $m_{st} = 3$. Using the correspondence between $RW(\sigma_w)$ and $RW(\sigma_{w'})$, the above argument can be applied to show that $m_{\tau(s) \tau(t)} = 3$, so $a_{ts} = a'_{\tau(t) \tau(s)}$ since $a_{st} = a'_{\tau(s) \tau(t)}$ by the argument in the previous paragraph. Finally, if $rst \in W^I$, then Lemma \ref{C2} implies that
\begin{equation*}
\delta_{rt} - a_{rt} + a_{rs} a_{st} = \delta_{\tau(r) \tau(t)} - a'_{\tau(r) \tau(t)} + a'_{\tau(r) \tau(s)} a'_{\tau(s) \tau(t)}.
\end{equation*}
Since $\tau$ is bijective, $\delta_{rt} = \delta_{\tau(r) \tau(t)}$. Also, applying the argument in the previous paragraph again, we have $a_{rt} = a'_{\tau(r) \tau(t)}$ and $a_{st} = a'_{\tau(s) \tau(t)} \neq 0$. Hence $a_{rs} = a'_{\tau(r) \tau(s)}$. In conclusion, we always have $a_{rs} = a'_{\tau(r) \tau(s)}$. Note that the ordered pair $(t_1, t_2) = (r, s)$ or $(s, r)$, and the latter one can only appear when $rs = sr$. If $(t_1, t_2) = (r, s)$, then $(\tau(t_1), \tau(t_2)) = (\tau(r), \tau(s))$, and hence $a_{t_1 t_2} = a'_{\tau(t_1) \tau(t_2)}$. Otherwise, $(t_1, t_2) = (s, r)$ and $(\tau(t_1), \tau(t_2)) = (\tau(s), \tau(r))$. Since $sr \notin [1, w]^I$, $\tau(s) \tau(r)$ does not lie in $[1, w']^{I'}$. Hence $\tau(s) \tau(r) = \tau(r) \tau(s)$. Therefore, $a_{t_1 t_2} = -1 = a'_{\tau(t_1) \tau(t_2)}$.
\end{proof}

\begin{proof}[Proof of Theorem \ref{thm-iso->car}]
    Assume that both $w$ and $w'$ are fully supported. If $X(w, A, I)$ and $X(w', A', I')$ are algebraically isomorphic, then $|I| = |I'|$ by Proposition \ref{not}. The case $I$ being empty is \cite[Theorem 1.3]{RS21}. The case $|I| = 1$ is Proposition \ref{neces}.
\end{proof}

\section{Constructing Isomorphism: Equal Supports}\label{con iso}
In this section, we will construct the isomorphism for Cartan equivalent pair.

Let $A = (a_{st})_{(s, t) \in S^2}$ and $A' = (a'_{s't'})_{(s', t') \in S'^2}$ be two Cartan matrices with associated Weyl groups $W$ and $W'$, and sets of simple reflections $S$ and $S'$, respectively. Let $I \subset S$ and $I' \subset S'$. Take two fully supported elements $w \in W^I$ and $w' \in W'^{I'}$. We assume that $(w, A, I)$ and $(w', A', I')$ are Cartan equivalent; in other words, there is a bijection $\tau: S \to S'$ sending $I$ to $I'$ such that:
\begin{enumerate}
\item for some reduced word $w = s_1 \cdots s_k$, $w' = \tau(s_1) \cdots \tau(s_k)$;
\item for any $t_1, t_2 \in S$, $a_{t_1 t_2} = a'_{\tau(t_1) \tau(t_2)}$ whenever $t_1 t_2 \leq w$.
\end{enumerate}

We want to show that the varieties $X(w, A, I)$ and $X(w', A', I')$ are algebraically isomorphic. First, we have the following result.

\begin{lemma}\label{Car equ}
The bijection $\tau: S \to S'$ induces a poset isomorphism $[1, w]^I \to [1, w']^{I'}$.
\end{lemma}

\begin{proof}
If $v = s_1 \cdots s_k \leq w$ is a reduced word, then we define $\tau(v) = \tau(s_1) \cdots \tau(s_k)$. By \cite[Lemma 2.2]{RS21}, for any $v \in [1, w]$, $\tau(v)$ is well-defined, and $\tau$ induces a bijection
\begin{align*}
RW(v) \to RW(\tau(v)), && (s_1, \cdots, s_k) \mapsto (\tau(s_1), \cdots, \tau(s_k)),
\end{align*}
for any $v \in [1, w]$. It follows that $\tau$ also induces a bijection
\begin{equation*}
D_R(v) \to D_R(\tau(v)).
\end{equation*}
\cite[Lemma 2.4]{RS21} implies that $\tau$ induces a poset isomorphism $[1, w] \to [1, w']$. Given $v \in [1, w]$, since
\begin{equation*}
v \in W^I \iff D_R(v) \cap I = \varnothing \iff D_R(\tau(v)) \cap I' = \varnothing \iff \tau(v) \in W^I,
\end{equation*}
the poset isomorphism $[1, w] \to [1, w']$ restricts to the isomorphism $[1, w]^I \to [1, w']^{I'}$.
\end{proof}

In the lemmas below, we add the additional assumption that $a_{st} \leq a'_{\tau(s) \tau(t)}$ for all simple reflections $s, t \in S$.

Let $\mathfrak{g}(A)$ be the Kac-Moody Lie algebra associated to the Cartan matrix $A$, and let 
\begin{equation*}
\mathfrak{g}(A) = \mathfrak{h}(A) \oplus \bigoplus_{\alpha \in R} \mathfrak{g}(A)_\alpha
\end{equation*}
be its root space decomposition. The bijection $\tau$ induces a bijection $\tau: \Delta \mapsto \Delta'$ between simple roots, given by $\alpha_s \mapsto \alpha_{\tau(s)}$. Then $\mathfrak{h}(A)$ is the Cartan subalgebra spanned by $h_\beta$ for $\beta \in \Delta$, which can be identified as the simple coroot $\beta^\vee = 2 \beta / \langle \beta, \beta \rangle$. We let $e_\alpha \in \mathfrak{g}(A)_\alpha$, $f_\alpha \in \mathfrak{g}(A)_{- \alpha}$, and $h_\beta \in \mathfrak{h}(A)$ be a Chevalley basis of $\mathfrak{g}(A)$, where $\alpha \in R^+$ and $\beta \in \Delta$. We set
\begin{equation*}
    \mathfrak{n}(A)^\pm = \bigoplus_{\alpha \in R^\pm} \mathfrak{g}(A)_\alpha
\end{equation*}
and
\begin{equation*}
    \mathfrak{b}(A) = \mathfrak{h}(A) \oplus \mathfrak{n}(A)^+.
\end{equation*}
The similar notations can be defined when replacing $A$ by $A'$. The lemma below is \cite[Lemma 3.5 (a)]{RS21}.

\begin{lemma}\label{Lie homomor}
There are surjective Lie algebra homomorphisms
\begin{equation*}
\varphi^\pm: \mathfrak{n}^\pm(A) \to \mathfrak{n}^\pm(A'),
\end{equation*}
given by $e_\alpha \mapsto e'_{\tau(\alpha)}$ and $f_\alpha \mapsto f'_{\tau(\alpha)}$, respectively.
\end{lemma}

For $v \in W$, we denote
\begin{equation*}
\mathfrak{n}(A)^+_v = \bigoplus_{\alpha \in I(v)} \mathfrak{g}(A)_\alpha,
\end{equation*}
where $\mathfrak{g}(A)_\alpha$ is the root space corresponding to the root $\alpha$. We define $\mathfrak{n}^+(A')_{v'}$ for $v' \in W'$ in a similar way. The following lemma rephrases \cite[Lemma 3.8]{RS21}.

\begin{lemma}\label{varphi}
The homomorphism $\varphi^+$ induces an isomorphism
\begin{equation*}
\varphi^+: \mathfrak{n}^+(A)_v \to \mathfrak{n}^+(A')_{\tau(v)}
\end{equation*}
for each $v \in [1, w]^I$.
\end{lemma}

Let $\lambda$ be a dominant integral weight of $\mathfrak{g}(A)$. Then there is a unique irreducible $\mathfrak{g}(A)$-module $L_\lambda$ with the highest weight $\lambda$. It can be constructed below. Let $\mathbb{C}_\lambda$ be a one-dimensional complex vector space spanned by $\omega = \omega_\lambda$, whose $\mathfrak{b}(A)$-module structure is given by
\begin{align*}
h \omega & = \lambda(h) \omega, && h \in \mathfrak{h}(A), \\
e \omega & = 0, && e \in \mathfrak{n}^+(A).
\end{align*}
Then the \emph{Verma module} with highest weight $\lambda$ is
\begin{equation*}
M_\lambda = U(\mathfrak{g}(A)) \otimes_{U(\mathfrak{b}(A))} \mathbb{C}_\lambda,
\end{equation*}
where $U(\mathfrak{g}(A))$ and $U(\mathfrak{b}(A))$ are the universal enveloping algebras of $\mathfrak{g}(A)$ and $\mathfrak{b}(A)$, respectively. It has a submodule, denoted by $M^1_\lambda$, which is generated by $f_\alpha^{\lambda(h_\alpha) + 1} \otimes 1 = f_\alpha^{\lambda(h_\alpha) + 1} \omega$. Then by \cite[Theorem 8.28]{Kir08}, we have
\begin{equation*}
L_\lambda = M_\lambda/M^1_\lambda.
\end{equation*}

We say that a dominant integral weight $\lambda$ is an \emph{$I$-regular weight} if, for any simple reflection $s \in S$, $\lambda(\alpha_s^\vee) = 0$ if and only $s \in I$. The following result is in \cite[Lemma 7.1.2]{Kum12}.

\begin{lemma}\label{clo imm}
Suppose that $\lambda$ is an $I$-regular weight, and $V$ is a finite-dimensional highest weight module with the highest weight vector $\omega$. Then there is a closed immersion from the flag variety $G/P = X(A, I)$ into the projective space $\mathbb{P}(V)$, given by $gP \mapsto [g \omega]$, where $[g \omega]$ denote the line through $g \omega$.
\end{lemma}

If $\lambda$ and $\lambda'$ are $I$ and $I'$-regular weights, then we write $V = L_\lambda$ and $V' = L_{\lambda'}$ with corresponding highest weight vectors $\omega$ and $\omega'$, respectively.

\begin{lemma}\label{pi}
Let $\lambda$ and $\lambda'$ be $I$ and $I'$-regular weights respectively. If $\lambda(h_\alpha) = \lambda'(h'_{\tau(\alpha)})$ for each $\alpha \in \Delta$, then there is a surjective $\mathfrak{n}^-(A)$-homomorphism $\pi: V \to V'$, sending the highest weight vector $\omega$ to $\omega'$, where $V'$ is regarded as a $\mathfrak{n}^-(A)$-module via the homomorphism $\varphi^-$. The homomorphism $\pi$ satisfies
\begin{equation*}
\pi(\exp(e) v \omega) = \exp(\varphi^+(e)) \tau(v) \omega'
\end{equation*}
for all $e \in \mathfrak{n}^+(A)$ and $v \leq w$.
\end{lemma}

\begin{remark}
The formula makes sense since the Weyl group element $v$ can be lifted to an element in the normalizer of the maximal torus, and different liftings acting on the vector $\omega$ give the same result.
\end{remark}

\begin{proof}[Proof of Lemma \ref{pi}]
The first part of the lemma is \cite[Lemma 3.5 (b)]{RS21}. The second part is \cite[Lemma 3.9]{RS21}.
\end{proof}

With the preparation above, we can prove Theorem \ref{thm-car->iso}.

\begin{proof}[Proof of Theorem \ref{thm-car->iso}]
    Using \cite[Lemma 4.8]{RS16}, we suppose that both $w$ and $w'$ are fully supported. Hence our assumption in this section holds.
    
    First, we assume that $a_{st} \leq a'_{s't'}$. Take $\lambda$ and $\lambda'$ to be $I$ and $I'$-regular weights, respectively. By possibly increasing $\lambda$ or $\lambda'$, we can assume that $\lambda(h_\alpha) = \lambda'(h'_{\tau(\alpha)})$ for each $\alpha \in \Delta$. Take $G$ (resp. $P$) to be the reductive group (resp. the parabolic subgroup) corresponding to the flag variety $X(A, I)$. Then a Schubert cell can be uniquely written as $U_v vP/P$ for some $v \in W^I$, where $U_v$ is the unipotent subgroup with Lie algebra $\mathfrak{n}^+(A)_v$. Hence under the closed immersion given in Lemma \ref{clo imm}, every element in $X(w, A, I)$ can be written uniquely as $[\exp(e) v \omega] \in \mathbb{P}(V)$, the line through $\exp(e) v \omega \in V$, for some $v \in [1, w]^I$ and $e \in \mathfrak{n}^+_v$. Then Lemma \ref{varphi} and \ref{pi} show that $\pi$ restricts to a bijection between Schubert cells indexed by $v$ and $\tau(v)$. Hence Lemma \ref{Car equ} implies that $\pi$ restricts to a bijection
    \begin{equation*}
        X(w, A, I) \to X(w', A', I').
    \end{equation*}
    Since Schubert varieties are normal, $X(w, A, I)$ and $X(w', A', I')$ are isomorphic.

    Next we remove the assumption $a_{st} \leq a'_{\tau(s) \tau(t)}$. Then we define a Cartan matrix $A'' = (a''_{st})$, where $a''_{st} = \max \{ a_{st}, a'_{\tau(s) \tau(t)} \}$. Compared to the Dynkin diagram of $A$, the one of $A''$ is just the diagram obtained by removing several edges, so it is still of finite type. Based on the argument above, we conclude that
    \begin{equation*}
        X(w, A, I) \cong X(w, A'', I) \cong X(w', A', I').
    \end{equation*}
\end{proof}

\section{Folding by Automorphisms}\label{folding-recall}
There are many isomorphic pairs of Schubert varieties which are not equally supported. For instance, \cite[Example 1.5]{RS21} claims that the Schubert varieties
\begin{align*}
X \left( s_3 s_2 s_1, \begin{pmatrix} 2 & -1 & 0 \\ -1 & 2 & -1 \\ 0 & -1 & 2 \end{pmatrix}, \{ s_2, s_3 \} \right) && \mbox{and} && X \left( s_1 s_2 s_1, \begin{pmatrix} 2 & -2 \\ -1 & 2 \end{pmatrix}, \{ s_2 \} \right),
\end{align*}
of types $A_3$ and $C_2$ respectively, are algebraically isomorphic to the projective space $\mathbb{P}^3$. Thus we need another method to deal with such pairs. We consider the pairs obtained from folding the root system.

We recall the basic definitions and results of folding a root system.

Let $R$ be a simply laced root system (of finite type, as usual) with simple roots $\Delta$ embedded in the real vector space $V = \mbox{Span } R$ with inner product $\langle \cdot, \cdot \rangle$. Let $\tau$ be a diagram automorphism of $(R, \Delta)$. It is a permutation of $\Delta$ (or the set $S$ of simple reflections) with the property
\begin{equation*}
\langle \alpha_{\tau(s)}, \alpha_{\tau(t)} \rangle = \langle \tau(\alpha_s), \tau(\alpha_t) \rangle
\end{equation*}
for both $s, t \in S$. If we extend the map $\alpha_s \mapsto \alpha_{\tau(s)}$ linearly, we may view $\tau$ as an isometry of $V$. Denote $\{ \Delta_i \}_i$ as the collection of $\tau$-orbits. It is known that the simple roots in the same $\tau$-orbit $\Delta_i$ are pairwise orthogonal. Let
\begin{equation*}
\beta_i^\vee = \sum_{\alpha \in \Delta_i} \alpha^\vee
\end{equation*}
be the sum of simple coroots in the same $\tau$-orbit, and let $\Delta^\tau$ be the collection of all such $\beta_i$. Then $\Delta^\tau$ is a set of simple roots of some root system $R^\tau$. The root system $(R^\tau, \Delta^\tau)$ is called the root system obtained from \emph{folding} the root system $(R, \Delta)$ by the automorphism $\tau$.

Figure \ref{fig:enter-label} gives some examples of folding.

\begin{figure}[h]
    \centering
    \includegraphics[width=1\linewidth]{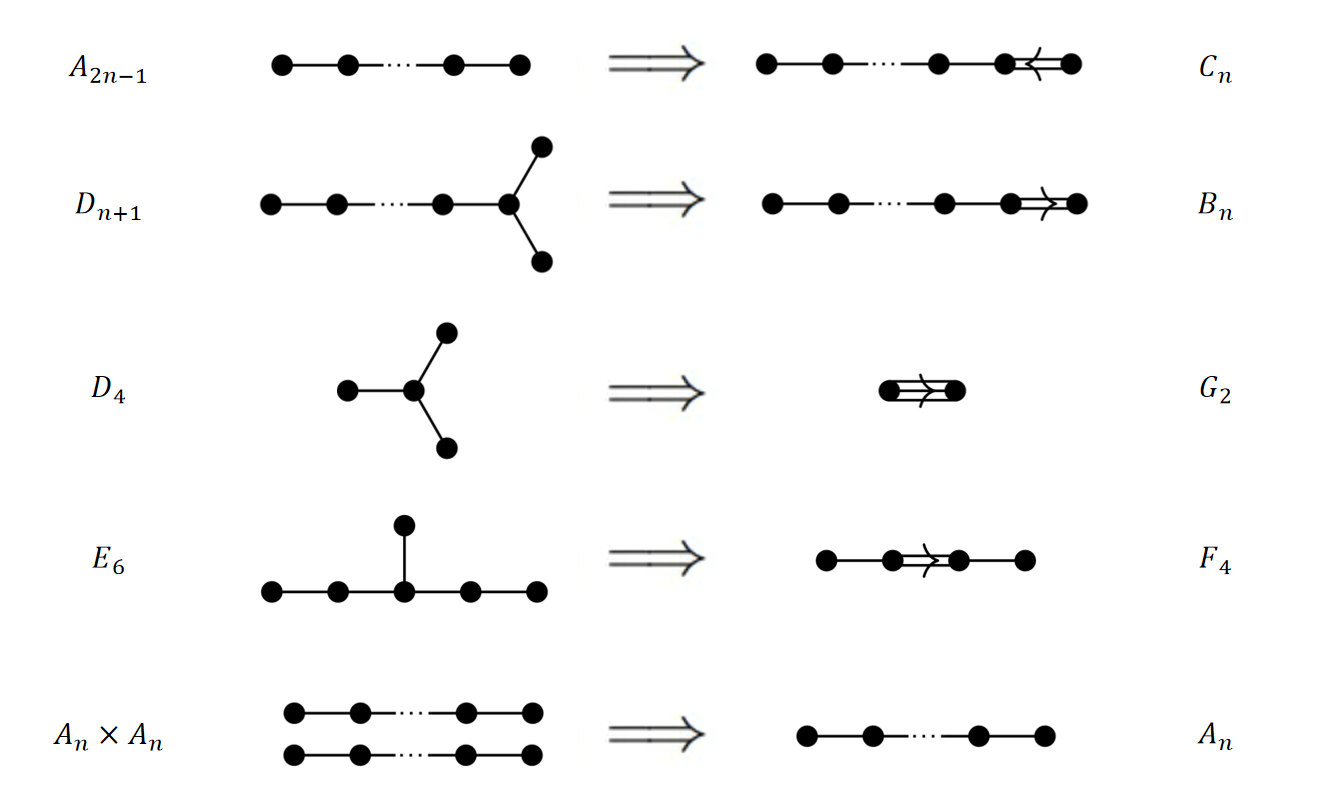}
    \caption{Examples of Folding}
    \label{fig:enter-label}
\end{figure}

\begin{lemma}
A sum of pairwise orthogonal coroots in a single $\tau$-orbit is a coroot of $R^\tau$. 
Conversely, all coroots of $R^\tau$ have this form.
\end{lemma}

This is \cite[Claim 1]{St08}.

Since $\tau s_\alpha \tau^{-1} = s_{\tau(\alpha)}$, $\tau$ acts via conjugation as an automorphism of the Weyl group $W$. Let $I_i$ be the collection of simple reflections $s_\alpha$ with $\alpha \in \Delta_i$. We denote
\begin{equation*}
t_i = \prod_{s \in I_i} s \in W.
\end{equation*}

\begin{lemma}\label{fold-weyl}
The homomorphism $s_{\beta_i} \mapsto t_i$ is an isomorphism from the Weyl group of the root system $R^\tau$ onto the subgroup $W^\tau$ of $W$ fixed by the action of $\tau$.
\end{lemma}

This is \cite[Claim 3]{St08}. In other words, denoting the collection of $t_i$ by $S^\tau$, $(W^\tau, S^\tau)$ is the Weyl group of the root system $(R^\tau, \Delta^\tau)$.

Recall that the entry $a_{st}$ of the Cartan matrix of $R$ is given by $a_{st} = \langle \alpha_s^\vee, \alpha_t \rangle$. We have $a_{st} = a_{\tau(s) \tau(t)}$.

\begin{lemma}\label{fold-cartan matrix}
Let
\begin{align*}
A = (a_{s_i s_j})_{(s_i, s_j) \in S^2} && \mbox{and} && A^\tau = (a_{t_k t_l})_{(t_k, t_l) \in {S^\tau}^2}
\end{align*}
be the Cartan matrices of the root systems $(R, \Delta)$ and $(R^\tau, \Delta^\tau)$, respectively. Then the formula
\begin{equation*}
a_{t_i t_j} = \sum_{s_i \in I_i} a_{s_i s_j}
\end{equation*}
holds for any $s_j \in I_j$.
\end{lemma}

\begin{proof}
By definition, we have
\begin{align*}
\beta_i^\vee = \sum_{\alpha_i \in \Delta_i} \alpha_i^\vee && \mbox{and} && \beta_j^\vee = \sum_{\alpha_j \in \Delta_j} \alpha_j^\vee.
\end{align*}
Recall that the simple roots in the same $\tau$-orbit $\Delta_i$ are pairwise orthogonal. We compute that
\begin{align*}
a_{t_i t_j} & = \langle \beta_i^\vee, \beta_j \rangle = \langle \beta_i^\vee, \beta_j^{\vee \vee} \rangle = \frac{2 \langle \beta_i^\vee, \beta_j^\vee \rangle}{\langle \beta_j^\vee, \beta_j^\vee \rangle} = \frac{2 \left\langle \sum_{\alpha_i \in \Delta_i} \alpha_i^\vee, \sum_{\alpha_j \in \Delta_j} \alpha_j^\vee \right\rangle}{\left\langle \sum_{\alpha_j \in \Delta_j} \alpha_j^\vee, \sum_{\alpha_j \in \Delta_j} \alpha_j^\vee \right\rangle} \\
& = \frac{2 \sum_{\alpha_i \in \Delta_i} \sum_{\alpha_j \in \Delta_j} \langle \alpha_i^\vee, \alpha_j^\vee \rangle}{\sum_{\alpha_j \in \Delta_j} \langle \alpha_j^\vee, \alpha_j^\vee \rangle} = \frac{2 \sum_{\alpha_i \in \Delta_i} \sum_{\alpha_j \in \Delta_j} \frac{2 \langle \alpha_i^\vee, \alpha_j \rangle}{\langle \alpha_j, \alpha_j \rangle}}{\sum_{\alpha_j \in \Delta_j} \left\langle \frac{2 \alpha_j}{\langle \alpha_j, \alpha_j \rangle}, \frac{2 \alpha_j}{\langle \alpha_j, \alpha_j \rangle} \right\rangle} \\
& = \frac{\sum_{\alpha_j \in \Delta_j} \frac{1}{\langle \alpha_j, \alpha_j \rangle} \sum_{\alpha_i \in \Delta_i} \langle \alpha_i^\vee, \alpha_j \rangle}{\sum_{\alpha_j \in \Delta_j} \frac{1}{\langle \alpha_j, \alpha_j \rangle}} = \frac{\sum_{\alpha_j \in \Delta_j} \frac{1}{\langle \alpha_j, \alpha_j \rangle} \sum_{s_i \in I_i} a_{s_i s_{\alpha_j}}}{\sum_{\alpha_j \in \Delta_j} \frac{1}{\langle \alpha_j, \alpha_j \rangle}}.
\end{align*}
To complete the proof, it is sufficient to show that the summation $\sum_{s_i \in I_i} a_{s_i s_j}$ is independent on $s_j \in I_j$. In fact,
\begin{equation*}
\sum_{s_i \in I_i} a_{s_i s_j} = \sum_{s_i \in I_i} a_{\tau(s_i) \tau(s_j)} = \sum_{s_i \in I_i} a_{s_i \tau(s_j)},
\end{equation*}
as required.
\end{proof}

Let $\mathfrak{g} = \mathfrak{n}^- \oplus \mathfrak{h} \oplus \mathfrak{n}^+$ and $\mathfrak{g}^\tau = {\mathfrak{n}^\tau}^- \oplus \mathfrak{h}^\tau \oplus {\mathfrak{n}^\tau}^+$ be two complex reductive Lie algebras corresponding to the root systems $(R, \Delta)$ and $(R^\tau, \Delta^\tau)$, respectively.

\begin{lemma}\label{Lie homomorphism}
There is an injective Lie algebra homomorphism
\begin{equation*}
\varphi: \mathfrak{g}^\tau \to \mathfrak{g},
\end{equation*}
given by $e_{t_i} \mapsto \sum_{s \in I_i} e_s$, $h_{t_i} \mapsto \sum_{s \in I_i} h_s$ and $f_{t_i} \mapsto \sum_{s \in I_i} f_s$, respectively.
\end{lemma}

\begin{proof}
First we show that $\varphi$ is a Lie algebra homomorphism. It suffices to check the Serre's relations; that is,
\begin{align*}
\varphi([h_{t_i}, h_{t_j}]) & = 0, \\
\varphi([h_{t_i}, e_{t_j}]) & = a_{t_i t_j} \varphi(e_{t_j}), \\
\varphi([h_{t_i}, f_{t_j}]) & = -a_{t_i t_j} \varphi(f_{t_j}), \\
\varphi([e_{t_i}, f_{t_j}]) & = \delta_{ij} h_{t_i}, \\
\varphi \left( (\mbox{ad } e_{t_i})^{1 - a_{t_i t_j}} e_{t_j} \right) & = 0, \\
\varphi \left( (\mbox{ad } f_{t_i})^{1 - a_{t_i t_j}} f_{t_j} \right) & = 0.
\end{align*}
We compute them one by one. For the first one,
\begin{equation*}
\varphi([h_{t_i}, h_{t_j}]) = \left[ \sum_{s_i \in I_i} h_{s_i}, \sum_{s_j \in I_j} h_{s_j} \right] = \sum_{s_i \in I_i} \sum_{s_j \in I_j} [h_{s_i}, h_{s_j}] = 0.
\end{equation*}
The second and the third ones are similar. In fact, we have
\begin{equation*}
\varphi([h_{t_i}, e_{t_j}]) = \sum_{s_i \in I_i} \sum_{s_j \in I_j} [h_{s_i}, e_{s_j}] = \sum_{s_i \in I_i} \sum_{s_j \in I_j} a_{s_i s_j} e_{s_j} = \sum_{s_j \in I_j} a_{t_i t_j} e_{s_j} = a_{t_i t_j} \varphi(e_{t_j})
\end{equation*}
and
\begin{equation*}
\varphi([h_{t_i}, f_{t_j}]) = - \sum_{s_i \in I_i} \sum_{s_j \in I_j} a_{s_i s_j} f_{s_j} = - a_{t_i t_j} \varphi(f_{t_j}).
\end{equation*}
The fourth identity follows from the calculation
\begin{equation*}
\varphi([e_{t_i}, f_{t_j}]) = \sum_{s_i \in I_i} \sum_{s_j \in I_j} [e_{s_i}, f_{s_j}] = \sum_{s_i \in I_i} \sum_{s_j \in I_j} \delta_{s_i s_j} h_{s_i} = \delta_{ij} \sum_{s_i \in I_i} h_{s_i} = \delta_{ij} \varphi(h_{t_i}).
\end{equation*}
Now we show the fifth equation. Since the simple roots in the same $\tau$-orbit $\Delta_i$ are pairwise orthogonal, the operators $\mbox{ad } e_{s_i}$ commute with each other. Hence the left-hand side
\begin{equation*}
\varphi \left( (\mbox{ad } e_{t_i})^{1 - a_{t_i t_j}} e_{t_j} \right) = \sum_{s_j \in I_j} \left( \mbox{ad } \sum_{s_i \in I_i} e_{s_i} \right)^{1 - a_{t_i t_j}} e_{s_j}
\end{equation*}
is equal to zero by the pigeonhole principle and Serre's relation
\begin{equation*}
(\mbox{ad } e_{s_i})^{1 - a_{s_i s_j}} e_{s_j} = 0.
\end{equation*}
Similarly, the sixth equation also holds. Thus $\varphi$ is indeed a Lie algebra homomorphism.

By definition, $\varphi|_{\mathfrak{h}^\tau}: \mathfrak{h}^\tau \to \mathfrak{h}$ is injective. It remains to show that $\varphi|_{{\mathfrak{n}^\tau}^\pm}: {\mathfrak{n}^\tau}^\pm \to \mathfrak{n}^\pm$ is also injective. For any $\beta \in {R^\tau}^+$, it can be written as $\beta = t(\beta_i)$ for some $t \in W^\tau$ and $\beta_i \in S^\tau$. If we write $\beta_i^\vee = \sum_{\alpha \in \Delta_i} \alpha^\vee$, then $\beta^\vee = \sum_{\alpha \in \Delta_i} t(\alpha^\vee)$. Each coroot $t(\alpha^\vee) \in R$ is positive since they lie in the same $\tau$-orbit, and $\tau$ sends positive (resp. negative) coroots to positive coroots (resp. negative). Again, by definition, the image of the root vector $e_\beta$ is sent into the subspace $\bigoplus_{\alpha \in \Delta_i} \mathfrak{g}_{t(\alpha)} \subset \mathfrak{n}^+$, and different $\beta \in {R^\tau}^+$ correspond to different subspaces, so $\varphi|_{{\mathfrak{n}^\tau}^+}: {\mathfrak{n}^\tau}^+ \to \mathfrak{n}^+$ is injective. Similarly, $\varphi|_{{\mathfrak{n}^\tau}^-}: {\mathfrak{n}^\tau}^- \to \mathfrak{n}^-$ is also injective.
\end{proof}

From the proof of Lemma \ref{Lie homomorphism}, we conclude that every coroot of $R^\tau$ can be \emph{uniquely} written as a sum of pairwise orthogonal coroots of $R$ in a single $\tau$-orbit. If $\beta^\vee = \sum_{\alpha} \alpha^\vee$, then we write the collection of such roots $\alpha$ as $\Delta_\beta$. It is well-defined.

\section{Schubert Pairs from Folding}\label{explicit type}
Suppose that $(R, \Delta)$ is a simply laced root system with the Weyl group $(W, S)$ and $\tau$ is a (nontrivial) diagram automorphism. The root system $(R^\tau, \Delta^\tau)$ with the Weyl group $(W^\tau, S^\tau)$ is obtained from folding by $\tau$. Choosing $I \subset S$ and $I^\tau \subset S^\tau$ as two proper subsets. Assume that they satisfy the condition in Theorem \ref{fold-theorem}. In this section, we will give several examples to illustrate Theorem \ref{fold-theorem}. Before doing so, we first prove a lemma that will help us find the Schubert pairs.

\begin{lemma}\label{lem-fold-thm-bij}
The bijection $[1, w^\tau]^{I^\tau} \to [1, w]^I$ in Theorem \ref{fold-theorem} is a poset isomorphism preserving length; that is, for every reflection $v \in [1, w^\tau]^{I^\tau}$, the lengths $\ell_{W^\tau}(v)$ and $\ell_W(v^{\min})$ are the same.
\end{lemma}

\begin{proof}
\cite[Proposition 2.5.1]{BB05} implies that the bijection is order-preserving. To show it preserves length, we prove the following stronger statement: for a reduced expression $t_1 \cdots t_k \in [1, w^\tau]^{I^\tau}$, its image has the reduced expression $s_1 \cdots s_k \in [1, w]^I$ with $s_i \in I_i$.

The case $k = 0$ is trivial. Suppose that $k > 0$ and $t_1 \cdots t_k \in [1, w^\tau]^{I^\tau}$ are reduced expressions. Then $t_2 \cdots t_k \in [1, w^\tau]^{I^\tau}$ is also a reduced expression. By induction, its image has the reduced expression $s_2 \cdots s_k \in [1, w]^I$ with $s_i \in I_i$. Hence
\begin{equation*}
\left( \prod_{s \in I_1} s \right) s_2 \cdots s_k = t_1 (t_2 \cdots t_k)^{\min} \geq (t_1 \cdots t_k)^{\min} = \left( \prod_{s \in I_1} s s_2 \cdots s_k \right)^{\min} > s_2 \cdots s_k.
\end{equation*}
It follows that $(t_1 \cdots t_k)^{\min}$ has the form $s_1 \cdots s_k$ with $s_1 \in I_1$, since it is the minimal element lying between $s_2 \cdots s_k$ and $\prod_{s \in I_1} s s_2 \cdots s_k$. This shows that our claim is true.

Using the claim, it is easy to see that the bijection is length-preserving.
\end{proof}

Now we turn to the construction Schubert pairs. Without loss of generality, we can suppose that $(R^\tau, \Delta^\tau)$ is irreducible since any Schubert variety from a reducible root system is a product of Schubert varieties from irreducible root systems. Throughout this section, we use the type of a Cartan matrix to indicate the matrix itself unless there is ambiguity.

If $(R^\tau, \Delta^\tau)$ is still simply laced, then $(R, \Delta)$ is reducible, and each connected component of its Dynkin diagram can be identified as the one of $(R^\tau, \Delta^\tau)$, so we may identify the Weyl group $W$ as $(W^\tau)^n$, where $n$ is the number of components. The action of automorphism $\tau$ is identifying each component as $(R^\tau, \Delta^\tau)$. We want to find the relation between $I$ and $I^\tau$. Suppose that $t \in S^\tau \setminus I^\tau$ is sent to $s \in S \setminus I$. Then every simple reflection $S(w)$ corresponds to the simple roots in the same connected component. Indeed, suppose that $w^\tau = t_1 \cdots t_k \in {W^\tau}^{I^\tau}$ and $w = s_1 \cdots s_k \in W^I$ are reduced words, where $s_i \in I_i$ by the proof of Lemma \ref{lem-fold-thm-bij}. Let $1 \leq i < k$ is the largest integer such that $s_i$ does not correspond to a simple root in the same connected component. Then the image of $t_i \cdots t_k \in [1, w^\tau]^{I^\tau}$ in $[1, w]^I$ is
\begin{equation*}
    s_i \cdots s_k = (s'_i s_i \cdots s_k)^{\min},
\end{equation*}
where $s'_i \in I_i$ is the unique simple reflection corresponding to a simple root in the same connected component as $s_k$. We have $s'_i \neq s_j$ for $i < j \leq k$. It follows that $s'_i s_i \cdots s_k \in [1, w]^I$ is a reduced word, which is a contradiction. Hence we have $I = I^\tau \sqcup S^\tau \sqcup \cdots \sqcup S^\tau$ and
\begin{equation*}
W^I = {W^\tau}^{I^\tau} \times ({W^\tau}^{S^\tau})^{n-1} = {W^\tau}^{I^\tau}.
\end{equation*}
Thus, the isomorphism given in Theorem \ref{fold-theorem} is trivial.

\begin{example}
Consider the pair of root systems $A_n \times A_n$ and $A_n$ (see Figure \ref{fig:AA-A}). Write $S = \{ s_1, \cdots s_n, s'_1, \cdots s'_n \}$ and $S^\tau = \{ t_1, \cdots t_n \}$. We take $I = \{ s'_1, \cdots s'_n \}$, for example. If $w \in W^I$, then $s'_i \notin S(w)$. Thus if we write $w = s_{i_1} \cdots s_{i_k} \in W^I$, then $w^\tau = t_{i_1} \cdots t_{i_k} \in {W^\tau}^{I^\tau}$. Let $G = SL_{n+1}(\mathbb{C})$, $B$ the standard Borel subgroup, and $P = B \times G$ the parabolic subgroup of type $A_n \times A_n$ corresponding to $I$. Then
\begin{equation*}
X(w, A, I) = (B \times B)wP/P \cong BwB/B \times BG/G \cong BwB/B = X(w^\tau, A^\tau, I^\tau).
\end{equation*}
\begin{figure}[h]
    \centering
    \includegraphics[width=1\linewidth]{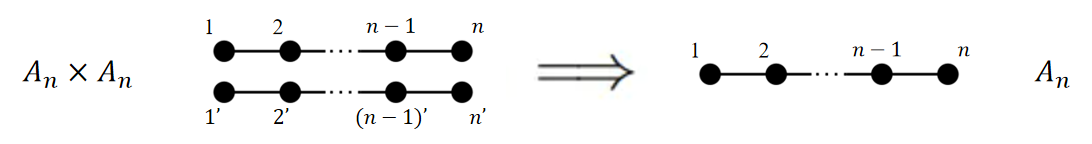}
    \caption{The Pair of Root Systems $A_n \times A_n$ and $A_n$}
    \label{fig:AA-A}
\end{figure}
\end{example}

Next, we assume that $(R^\tau, \Delta^\tau)$ is not simply laced. Restricting $\tau$ to a component of the Dynkin diagram of $(R, \Delta)$, we may suppose that $(R, \Delta)$ is irreducible. There are exactly four cases. We study them one by one.

\begin{example}\label{A-C-fold}
Consider the pair of root systems $A_{2n-1}$ and $C_n$ (see Figure \ref{fig:A-C}). Let $I = S \setminus \{ s_1 \}$ and $I^\tau = S^\tau \setminus \{ t_1 \}$. Both the posets $W^I$ and ${W^\tau}^{I^\tau}$ are chains, whose poset structures are given by
\begin{equation*}
W^I = \{ 1 < s_1 < s_2 s_1 < \cdots < s_{2n-1} s_{2n-2} \cdots s_1 \}
\end{equation*}
and
\begin{equation*}
{W^\tau}^{I^\tau} = \{ 1 < t_1 < t_2 t_1 < \cdots < t_n \cdots t_1 < t_{n-1} t_n \cdots t_1 < t_1 \cdots t_n \cdots t_1 \}.
\end{equation*}
Hence the canonical map ${W^\tau}^{I^\tau} \to W^I$ is a bijection. By Theorem \ref{fold-theorem} and \cite[Lemma 4.8]{RS16}, we have the isomorphisms
\begin{align*}
X(t_k \cdots t_1, C_n, S^\tau \setminus \{ t_1 \}) & \cong X(s_k \cdots s_1, A_{2n-1}, S \setminus \{ s_1 \}), \\
X(t_k \cdots t_n \cdots t_1, C_n, S^\tau \setminus \{ t_1 \}) & \cong X(s_{2n-k} \cdots s_n \cdots s_1, A_{2n-1}, S \setminus \{ s_1 \}),
\end{align*}
for $k \leq n$.
\begin{figure}[h]
    \centering
    \includegraphics[width=1\linewidth]{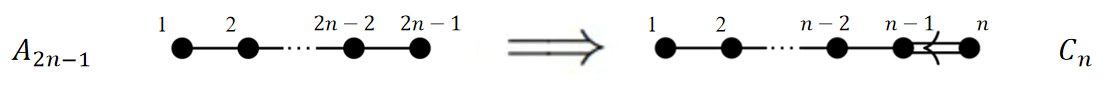}
    \caption{The Pair of Root Systems $A_{2n-1}$ and $C_n$}
    \label{fig:A-C}
\end{figure}
\end{example}

\begin{example}
Consider the pair of root systems $D_{n+1}$ and $B_n$ (see Figure \ref{fig:D-B}). The poset structures of $W^I$ and ${W^\tau}^{I^\tau}$ are messy, so we only give an example of the pair $D_4$ and $B_3$, with $I = \{ s_1, s_2, s_3 \}$ and $I^\tau = \{ t_1, t_2 \}$. The poset structures of $[1, t_1 t_3 t_2 t_3]^{I^\tau}$ and $[1, s_1 s_3 s_2 s_4]^I$ are given below,
\[\begin{tikzcd}
	& {t_1 t_3 t_2 t_3} &&&& {s_1 s_3 s_2 s_4} \\
	{t_1 t_2 t_3} && {t_3 t_2 t_3} && {s_1 s_2 s_4} && {s_3 s_2 s_4} \\
	& {t_2 t_3} &&&& {s_2 s_4} \\
	& {t_3} &&&& {s_4} \\
	& 1 &&&& 1
	\arrow[no head, from=2-1, to=1-2]
	\arrow[no head, from=2-3, to=1-2]
	\arrow[no head, from=2-5, to=1-6]
	\arrow[no head, from=2-7, to=1-6]
	\arrow[no head, from=3-2, to=2-1]
	\arrow[no head, from=3-2, to=2-3]
	\arrow[no head, from=3-6, to=2-5]
	\arrow[no head, from=3-6, to=2-7]
	\arrow[no head, from=4-2, to=3-2]
	\arrow[no head, from=4-6, to=3-6]
	\arrow[no head, from=5-2, to=4-2]
	\arrow[no head, from=5-6, to=4-6]
\end{tikzcd}\]
which are identical. Hence Theorem \ref{fold-theorem} implies the isomorphism
\begin{equation*}
X(t_1 t_3 t_2 t_3, B_3, S^\tau \setminus \{ t_3 \}) \cong X(s_1 s_3 s_2 s_4, D_4, S \setminus \{ s_4 \}).
\end{equation*}
\begin{figure}[h]
    \centering
    \includegraphics[width=1\linewidth]{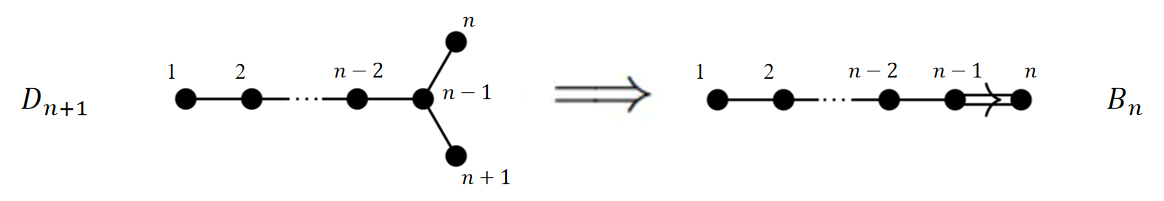}
    \caption{The Pair of Root Systems $D_{n+1}$ and $B_n$}
    \label{fig:D-B}
\end{figure}
\end{example}

\begin{example}
Consider the pair of root systems $D_4$ and $G_2$ (see Figure \ref{fig:D-G}). We let $I = \{ s_2, s_3, s_4 \}$ and $I^\tau = \{ t_2 \}$. In $G_2$, we have the poset $t_1 < t_2 t_1 < t_1 t_2 t_1$ in ${W^\tau}^{I^\tau}$. But
\begin{equation*}
t_1 t_2 t_1 = (s_1 s_2 s_3) s_4 (s_1 s_2 s_3) = (s_2 s_3) (s_1 s_4 s_1) (s_2 s_3) = (s_2 s_3) (s_4 s_1 s_4) (s_2 s_3) \mapsto s_2 s_3 s_4 s_1 \in W^I,
\end{equation*}
and $t_1 t_2 t_1 \in {W^\tau}^{I^\tau}$, $s_2 s_3 s_4 s_1 \in W^I$ having different lengths. Hence Theorem \ref{fold-theorem} only gives two isomorphisms in this case, which are
\begin{align*}
X(s_1, A, I) \cong X(t_1, A^\tau, I^\tau) && \mbox{and} && X(s_4 s_1, A, I) \cong X(t_2 t_1, A^\tau, I^\tau).
\end{align*}
The first isomorphism is trivial since any Schubert curve is isomorphic to the projective line $\mathbb{P}^1$. For the second one, the Schubert surfaces above are isomorphic to the projective plane $\mathbb{P}^2$ by the discussion in Example \ref{12}.
\begin{figure}[h]
    \centering
    \includegraphics[width=1\linewidth]{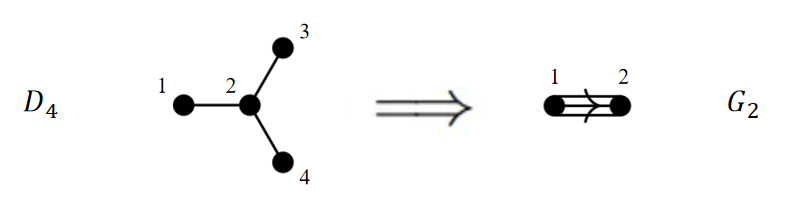}
    \caption{The Pair of Root Systems $D_4$ and $G_2$}
    \label{fig:D-G}
\end{figure}
\end{example}

\begin{example}
Consider the pair of root systems $E_6$ and $F_4$ (see Figure \ref{fig:E-F}). Let $I = S \setminus \{ s_1 \}$ and $I^\tau = S^\tau \setminus \{ t_1 \}$. The canonical map ${W^\tau}^{I^\tau} \to W^I$ gives more isomorphisms by Theorem \ref{fold-theorem}. We will not list them here.
\begin{figure}[h]
    \centering
    \includegraphics[width=1\linewidth]{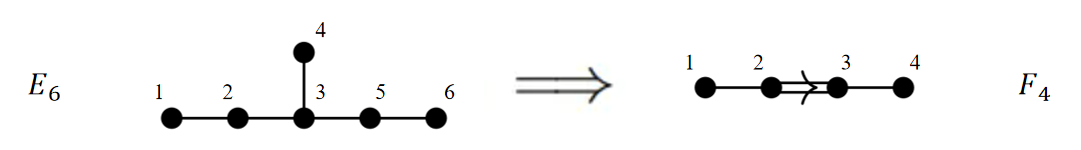}
    \caption{The Pair of Root Systems $E_6$ and $F_4$}
    \label{fig:E-F}
\end{figure}
\end{example}

Now we turn to studying the Schubert three-folds. A \emph{Schubert three-fold} is a Schubert variety of dimension three.

\begin{proof}
By \cite[Lemma 4.8]{RS16}, it is sufficient to consider the Schubert three-folds $(w, A, I)$ with fully supported $w$. Then the number of simple reflections in $I$ is at most three.

First, we consider the Schubert three-folds from reducible root systems. Then it is obvious that such a Schubert three-fold is a product of a Schubert surface and a projective line $\mathbb{P}^1$ since any Schubert curve is isomorphic to $\mathbb{P}^1$. By Theorem \ref{cla}, there are 7 isomorphism classes.

Next, suppose that the Schubert three-folds $(w, A, I)$ come from irreducible root systems. If $A$ has rank 2, then it has type $A_2$, $B_2$, or $G_2$. As what we did in Section \ref{EC}, for $n = 1, 2, 3$, we denote
\begin{equation*}
M_n = \begin{pmatrix} 2 & -1 \\ -n & 2 \end{pmatrix}.
\end{equation*}
By Theorem \ref{thm-iso->car}, any two of the 10 Schubert three-folds
\begin{align*}
X(s_1 s_2 s_1, M_1, \varnothing), & X(s_1 s_2 s_1, M_1, \{ s_2 \}), \\
X(s_1 s_2 s_1, M_2, \varnothing), X(s_2 s_1 s_2, M_2, \varnothing), & X(s_1 s_2 s_1, M_2, \{ s_2 \}), X(s_1 s_2 s_1, M_2, \{ s_1 \}), \\
X(s_1 s_2 s_1, M_3, \varnothing), X(s_2 s_1 s_2, M_3, \varnothing), & X(s_1 s_2 s_1, M_3, \{ s_2 \}), X(s_1 s_2 s_1, M_3, \{ s_1 \})
\end{align*}
are not isomorphic to each other.

Finally, we study the Schubert three-folds $(w, A, I)$ from irreducible root systems of rank 3. Then $A$ has type $A_3$, $B_3$, or $C_3$. We denote the Cartan matrices
\begin{align*}
N_1 = \begin{pmatrix} 2 & -1 & 0 \\ -1 & 2 & -1 \\ 0 & -1 & 2 \end{pmatrix}, && N_2 = \begin{pmatrix} 2 & -1 & 0 \\ -1 & 2 & -1 \\ 0 & -2 & 2 \end{pmatrix}, && N_3 = \begin{pmatrix} 2 & -1 & 0 \\ -1 & 2 & -2 \\ 0 & -1 & 2 \end{pmatrix}.
\end{align*}
The element $w \in W^I$ can be $s_1 s_2 s_3$, $s_1 s_3 s_2$, $s_2 s_1 s_3$, or $s_3 s_2 s_1$. If $I$ is empty, then there are 6 isomorphism classes
\begin{align*}
X(s_1 s_2 s_3, N_1, \varnothing), & X(s_1 s_2 s_3, N_3, \varnothing), \\
X(s_1 s_3 s_2, N_1, \varnothing), & X(s_1 s_3 s_2, N_2, \varnothing), \\
X(s_2 s_1 s_3, N_1, \varnothing), & X(s_2 s_1 s_3, N_3, \varnothing).
\end{align*}
Here, we note that there are isomorphisms
\begin{align*}
X(s_1 s_2 s_3, N_1, \varnothing) \cong X(s_3 s_2 s_1, N_1, \varnothing) && \mbox{and} && X(s_1 s_2 s_3, N_3, \varnothing) \cong X(s_3 s_2 s_1, N_2, \varnothing).
\end{align*}
If $I$ is a singleton, then there are 8 isomorphism classes
\begin{align*}
X(s_1 s_2 s_3, N_1, \{ s_1 \}), X(s_1 s_2 s_3, N_3, \{ s_1 \}), & X(s_1 s_2 s_3, N_1, \{ s_2 \}), X(s_1 s_2 s_3, N_3, \{ s_2 \}), \\
X(s_1 s_3 s_2, N_1, \{ s_1 \}), & X(s_1 s_3 s_2, N_2, \{ s_1 \}), \\
X(s_2 s_1 s_3, N_1, \{ s_2 \}), & X(s_2 s_1 s_3, N_3, \{ s_2 \}).
\end{align*}
If $|I| = 2$, then there are at most 4 isomorphism classes
\begin{align*}
X(s_1 s_2 s_3, N_1, \{ s_1, s_2 \}), & X(s_1 s_2 s_3, N_3, \{ s_1, s_2 \}), \\
X(s_1 s_3 s_2, N_1, \{ s_1, s_3 \}), & X(s_1 s_3 s_2, N_2, \{ s_1, s_3 \}).
\end{align*}

From Example \ref{A-C-fold} or \cite[Example 1.5]{RS21}, we know that the Schubert varieties
\begin{align*}
X(s_1 s_2 s_3, N_1, \{ s_1, s_2 \}) && \mbox{and} && X(s_2 s_1 s_2, M_2, \{ s_1 \})
\end{align*}
are isomorphic. Hence there are at most $7 + 10 + 6 + 8 + 4 - 1 = 34$ Schubert three-folds up to isomorphism.
\end{proof}

\section{Constructing Isomorphism: Folding}\label{con iso 2}
In this section, we will prove Theorem \ref{fold-theorem}. We suppose that the assumption in the statement of the theorem holds and we will use the notation given in the statement.

As we did in the equally supported case, if $\lambda$ and $\lambda^\tau$ are $I$ and $I^\tau$-regular weights, then we write $V = L_\lambda$ and $V^\tau = L_{\lambda^\tau}$ with corresponding highest weight vectors $\omega$ and $\omega^\tau$, respectively.

\begin{lemma}\label{fold-pi}
Let $\lambda$ and $\lambda^\tau$ be $I$ and $I^\tau$-regular weights, respectively. If $\lambda^\tau(h_{t_i}) = \sum_{s \in I_i} \lambda(h_s)$ for each $t_i \in S^\tau$, then there is a $\mathfrak{g}^\tau$-homomorphism $\pi: V^\tau \to V$, sending $\omega^\tau$ to $\omega$, where $V$ is regarded as an $\mathfrak{g}^\tau$-module via the homomorphism $\varphi$. The homomorphism $\pi$ satisfies
\begin{equation*}
\pi(\exp(e)v \omega^\tau) = \exp(\varphi(e)) v^{\min} \omega
\end{equation*}
for all $e \in {\mathfrak{n}^\tau}^+$ and $v \in [1, w^\tau]^{I^\tau}$.
\end{lemma}

\begin{proof}
Sending $\omega^\tau$ to $\omega$ induces an $\mathfrak{g}^\tau$-homomorphism $\pi_0: M_{\lambda^\tau} \to M_\lambda$. Recall that $V = M_\lambda/M_\lambda^1$ and $V^\tau = M_{\lambda^\tau}/M_{\lambda^\tau}^1$, where $M_\lambda^1$ and $M_{\lambda^\tau}^1$ are generated by $f_s^{\lambda(h_s) + 1} \omega$ and $f_t^{\lambda^\tau(h_t) + 1} \omega^\tau$, respectively. Notice that
\begin{equation*}
\pi_0 \left( f_{t_i}^{\lambda^\tau(h_{t_i}) + 1} \omega^\tau \right) = \left( \sum_{s \in I_i} f_s \right)^{\sum_{s \in I_i} \lambda(h_s) + 1} \omega \in M_\lambda^1
\end{equation*}
by pigeonhole principle. Hence $\pi_0$ induces an ${\mathfrak{n}^\tau}^-$-homomorphism $\pi: V^\tau \to V$, as required.

Now we prove the required formula. First, both the action of $e \in {\mathfrak{n}^\tau}^+$ and the action $f \in {\mathfrak{n}^\tau}^-$ on $V^\tau$, so we have
\begin{align*}
\pi(\exp(e) \eta_1) = \exp(\varphi(e)) \pi(\eta_1) && \mbox{and} && \pi(\exp(f) \eta_1) = \exp(\varphi(f)) \pi(\eta_1)
\end{align*}
for any $\eta_1 \in V^\tau$. The simple reflection $t_i \in S^\tau$ on $V$ acts by $\exp(f_{t_i}) \exp(-e_{t_i}) \exp(f_{t_i})$. Applying the above formulas, for each $\eta_2 \in V^\tau$, we compute that
\begin{align*}
\pi(t_i \eta_2) & = \pi(\exp(f_{t_i}) \exp(-e_{t_i}) \exp(f_{t_i}) \eta_2) = \exp(\varphi(f_{t_i})) \exp(\varphi(-e_{t_i})) \exp(\varphi(f_{t_i})) \pi(\eta_2) \\
& = \exp \left( \sum_{s_i \in I_i} f_{s_i} \right) \exp \left( - \sum_{s_i \in I_i} e_{s_i} \right) \exp \left( \sum_{s_i \in I_i} f_{s_i} \right) \pi(\eta_2) \\
& = \prod_{s_i \in I_i} \exp(f_{s_i}) \prod_{s_i \in I_i} \exp(-e_{s_i}) \prod_{s_i \in I_i} \exp(f_{s_i}) \pi(\eta_2) \\
& = \prod_{s_i \in I_i} (\exp(f_{s_i}) \exp(-e_{s_i}) \exp(f_{s_i})) \pi(\eta_2) = \prod_{s_i \in I_i} s_i \pi(\eta_2) = t_i \pi(\eta_2).
\end{align*}
Thus, the $\mathfrak{g}^\tau$-homomorphism $\pi$ commutes with the Weyl group $W^\tau$-action. Combining these two results, we conclude that
\begin{equation*}
\pi(\exp(e)v \omega^\tau) = \exp(\varphi(e)) v \omega = \exp(\varphi(e)) v^{\min} \omega
\end{equation*}
since $\omega$ is the highest weight vector corresponding to the $I$-regular weight $\lambda$.
\end{proof}

For $v \in [1, w^\tau]^{I^\tau}$, we denote the root space decompositions
\begin{align*}
{\mathfrak{n}^\tau_v}^+ = \bigoplus_{\beta \in I_{W^\tau}(v)} \mathfrak{g}^\tau_\beta && \mbox{and} && \mathfrak{n}_{v^{\min}}^+ = \bigoplus_{\alpha \in I_W(v^{\min})} \mathfrak{g}_\alpha.
\end{align*}

\begin{lemma}\label{fold-bijective}
Let $v \in [1, w^\tau]^{I^\tau}$. The homomorphism $\pi$ defines a bijection
\begin{equation*}
\exp({\mathfrak{n}^\tau_v}^+)v \omega^\tau \to \exp(\mathfrak{n}_{v^{\min}}^+) v^{\min} \omega.
\end{equation*}
\end{lemma}

\begin{proof}
First, we show that $\pi(\exp({\mathfrak{n}^\tau_v}^+)v \omega^\tau) \subset \exp(\mathfrak{n}_{v^{\min}}^+) v^{\min} \omega$. Take a nonzero root vector $e_\beta \in \mathfrak{g}^\tau_\beta$. Then $\varphi(e_\beta) = \sum_{\alpha \in \Delta_\beta} e_\alpha$ with $e_\alpha \in \mathfrak{g}_\alpha$ by the proof of Lemma \ref{Lie homomorphism}. Then for constants $a_\beta$, we have
\begin{equation*}
    \exp \left( \sum_{\beta \in I_{W^\tau}(v)} a_\beta e_\beta \right) v \omega^\tau \mapsto \exp \left( \sum_{\beta \in I_{W^\tau}(v)} a_\beta \sum_{\alpha \in \Delta_\beta} e_\alpha \right) v^{\min} \omega.
\end{equation*}
We denote the exponential on the right by $E$. Let $\alpha_1 \in R^+ \setminus I_W(v^{\min})$ be a root such that $e_{\alpha_1}$ appearing in the exponential above with minimal height and let $c_1$ be its coefficient. Then any root vector $e_\alpha$ with $\alpha \in R^+ \setminus I_W(v^{\min})$ appearing in $E \exp(- c_1 e_{\alpha_1})$ (writing it as a single exponential using a single exponential by Campbell-Hausdorff formula) satisfies that $\alpha$ has height no less than the height of $\alpha_1$, and the number of roots with the same height as $\alpha_1$ is strictly less, which can be seen by Campbell-Hausdorff formula. Applying the algorithm several times, the minimal height of the roots appearing in $E \exp(- c_1 e_{\alpha_1}) \cdots \exp(- c_k e_{\alpha_k})$ is strictly larger than the height of $\alpha_1$. Therefore, repeat the algorithm finite times, there will not be any root appearing in
\begin{equation*}
    E \exp(- c_1 e_{\alpha_1}) \cdots \exp(- c_n e_{\alpha_n}) =: E'
\end{equation*}
lying in $R^+ \setminus I_W(v^{\min})$ since the height of a root is bounded. In other words, $E' \in \exp(\mathfrak{n}_{v^{\min}}^+)$. We also note that each $\alpha_i \in R^+ \setminus I_W(v^{\min})$ by construction. Therefore,
\begin{align*}
Ev^{\min} \omega & = E' \exp(c_n e_{\alpha_n}) \cdots \exp(c_1 e_{\alpha_1}) v^{\min} \omega \\
& = E' v^{\min} \cdot \exp \left( c_n e_{{v^{\min}}^{-1}(\alpha_n)} \right) \cdots \exp \left( c_1 e_{{v^{\min}}^{-1}(\alpha_1)} \right) \omega \\
& = E' v^{\min} \omega \in \exp(\mathfrak{n}_{v^{\min}}^+) v^{\min} \omega,
\end{align*}
since every $e_\alpha$ acts on $\omega$ is zero.

This algorithm gives a morphism between the affine spaces
\begin{align*}
    \rho: {\mathfrak{n}^\tau_v}^+ \to \mathfrak{n}_{v^{\min}}^+, && \sum_{\beta \in I_{W^\tau}(v)} a_\beta e_\beta \mapsto \sum_{\alpha \in I_W(v^{\min})} a_\alpha e_\alpha.
\end{align*}
The morphism $\rho$ is independent on the construction process. Indeed, different elements $e, e' \in \mathfrak{n}_{v^{\min}}^+$ define different elements $\exp(e) v^{\min} \omega$, $\exp(e') v^{\min} \omega$.

Next, we show that $\rho$ is injective. Suppose that there exists $b_\beta, c_\beta$ such that $\sum_{\beta \in I_{W^\tau}(v)} b_\beta e_\beta$ and $\sum_{\beta \in I_{W^\tau}(v)} c_\beta e_\beta$ have the same image. Then
\begin{align*}
    \exp \left( \sum_{\beta \in I_{W^\tau}(v)} b_\beta e_\beta \right) v \omega^\tau && \mbox{and} && \exp \left( \sum_{\beta \in I_{W^\tau}(v)} c_\beta e_\beta \right) v \omega^\tau
\end{align*}
have the same image under $\pi$. There exist constants $a_\beta$ such that
\begin{equation*}
    \exp \left( \sum_{\beta \in I_{W^\tau}(v)} b_\beta e_\beta \right) \exp \left( \sum_{\beta \in I_{W^\tau}(v)} c_\beta e_\beta \right)^{-1} = \exp \left( \sum_{\beta \in I_{W^\tau}(v)} a_\beta e_\beta \right)
\end{equation*}
by Campbell-Hausdorff formula. Hence
\begin{equation*}
    \exp \left( \sum_{\beta \in I_{W^\tau}(v)} a_\beta e_\beta \right) v \omega^\tau \mapsto v^{\min} \omega
\end{equation*}
under $\pi$. Recall that $v^{\min} \omega = v \omega$, so
\begin{equation*}
    v^{-1} \exp \left( \sum_{\beta \in I_{W^\tau}(v)} a_\beta e_\beta \right) v \omega^\tau \mapsto \omega
\end{equation*}
under $\pi$, as $W^\tau$-action commutes with $\pi$ (see the proof of Lemma \ref{fold-pi}). Since $v^{-1}(\beta)$ is a negative root for $\beta \in I_{W^\tau}(v)$, we see that
\begin{equation*}
    v^{-1} \exp \left( \sum_{\beta \in I_{W^\tau}(v)} a_\beta e_\beta \right) v \omega^\tau = \exp(f) \omega^\tau,
\end{equation*}
where $f \in {\mathfrak{n}^\tau}^-$ is a nonzero element. Its image under $\pi$ is $\exp(f') \omega \neq \omega$, where $f' \in \mathfrak{n}^-$ is nonzero. Therefore, $\rho$ is injective.

According to Ax-Grothendieck theorem, $\rho$ is bijective. Hence the morphism
\begin{equation*}
\exp({\mathfrak{n}^\tau_v}^+)v \omega^\tau \to \exp(\mathfrak{n}_{v^{\min}}^+) v^{\min} \omega
\end{equation*}
is also a bijection.
\end{proof}

\begin{proof}[Proof of Theorem \ref{fold-theorem}]
Take $\lambda$ and $\lambda^\tau$ to be $I$ and $I^\tau$-regular weights, respectively, and assume that $\lambda(h_\alpha) = \lambda'(h'_{\tau(\alpha)})$ for each $\alpha \in \Delta$. Take $G$ (resp. $P$) to be the reductive group (resp. the parabolic subgroup) corresponding to the flag variety $X(w^\tau, A^\tau, I^\tau)$. Then a Schubert cell can be uniquely written as $U_v vP/P$ for some $v \in {W^\tau}^{I^\tau}$, where $U_v$ is the unipotent subgroup with Lie algebra ${\mathfrak{n}_v^\tau}^+$. By Lemma \ref{clo imm}, every element in $X(w^\tau, A^\tau, I^\tau)$ can be written uniquely as $[\exp(e) v \omega^\tau] \in \mathbb{P}(V^\tau)$, the line through $\exp(e) v \omega^\tau \in V^\tau$, for some $v \in [1, w^\tau]^{I^\tau}$ and $e \in {\mathfrak{n}^\tau}^+_v$. Then Lemma \ref{fold-pi} yields that the $\mathfrak{g}^\tau$-homomorphism $\pi$ defines a morphism from $X(w^\tau, A^\tau, I^\tau)$ to $X(w, A, I)$. By Lemma \ref{fold-bijective}, $\pi$ is bijective on each pair of Borel cells $\exp({\mathfrak{n}^\tau_v}^+)v \omega^\tau$ and $\exp(\mathfrak{n}_{v^{\min}}^+) v^{\min} \omega$. This gives the bijection between Schubert varieties $X(w^\tau, A^\tau, I^\tau)$ and $X(w, A, I)$. Therefore, the normality of Schubert varieties implies that $X(w^\tau, A^\tau, I^\tau)$ are $X(w, A, I)$ isomorphic as algebraic varieties.
\end{proof}

\printbibliography

\end{document}